\documentclass{scrartcl}

\usepackage[a4paper,margin=2.5cm]{geometry}

\usepackage[utf8]{inputenc}
\usepackage[T1]{fontenc}

\usepackage{graphicx}

\usepackage{amsmath,amssymb,amsthm}
\usepackage{enumitem}

\usepackage[dvipsnames,svgnames]{xcolor} 
\usepackage{caption}
\usepackage{subcaption}

\usepackage{cancel} 

\usepackage[style=alphabetic,maxalphanames=4,maxnames=4]{biblatex}
\AtEveryBibitem{\clearfield{issn}}

\newtheorem{thm}{Theorem}
\newtheorem{prop}[thm]{Proposition}
\newtheorem{lem}[thm]{Lemma}
\newtheorem{cor}[thm]{Corollary}
\theoremstyle{definition}
\newtheorem{Def}[thm]{Definition}

\theoremstyle{remark}
\newtheorem{rem}[thm]{Remark}

\newcommand{\dpart}[2]{\frac{\partial #1}{\partial #2}}

\newcommand{\Z}{\mathbb{Z}}
\newcommand{\Q}{\mathbb{Q}}
\newcommand{\K}{\mathbb{K}}
\newcommand{\R}{\mathcal{R}}

\usepackage[
    colorlinks=true,
    citecolor=NavyBlue,
    linkcolor=NavyBlue,   
    urlcolor=blue
]{hyperref}

\title{Enumeration of plane hypermaps with a mixed boundary I}
\author{J\'er\'emie Bouttier%
  \thanks{Sorbonne Université and Université Paris Cité, CNRS,
    IMJ-PRG, F-75005 Paris, France} \and%
  Bertrand Eynard%
  \thanks{Université Paris-Saclay, CNRS, CEA, Institut de physique
    théorique, 91191, Gif-sur-Yvette, France} \and Thomas
  Lejeune\footnotemark[1]
}
\date{\today}

\begin{document}
\maketitle

\vspace{-0.9cm}
\begin{abstract}

Plane hypermaps are plane maps endowed with a proper coloration
 of their inner faces in black or white. We consider the problem of enumerating 
plane hypermaps with prescribed face degrees and a $k$-alternating boundary condition: 
by this we mean the colors of inner faces incident to the outer face 
alternates at most $2k$ times when turning around the hypermap. 
 The present paper deals with the cases $k=1,2$, the general case being left 
 to the forthcoming part II.
Our approach relies on the so-called slice decomposition 
and uses crucially the notion of \emph{accessibility}, 
which exploits the canonical orientation of hypermaps and the marking variable~$t$ associated with 
 vertices, to enumerate pointed hypermaps by decomposing them according to the set 
 of all vertices that can access to the marked vertex.  
This process enables us to express the generating functions of hypermaps with mixed 
boundaries in terms of the generating functions of hypermap slices and to recover, in 
a purely combinatorial way, some formulas previously obtained through 
algebraic methods.

\end{abstract}

\vspace{-0.5cm}
\tableofcontents

\section{Introduction}

\subsection{Context and motivations}

The study of maps, i.e.\ graphs drawn on surfaces, is a vibrant topic
that connects combinatorics, probability theory, and theoretical
physics, to name a few. See for
instance~\cite{Lando2004,Schaeffer2015,Eynard2016,Curien2023} and
references therein. In this paper, we pursue the combinatorial study
of hypermaps via the so-called slice decomposition, initiated
in~\cite{Albenque2026}, with the goal of finding bijective proofs of the
intriguing formulas given in~\cite[Chapter~8]{Eynard2016}.

A hypermap can be seen as a properly face-bicolored map. Counting
hypermaps with prescribed numbers of faces of each degree and color is
an enumerative problem intimately connected with the so-called
two-matrix model and the Ising model on random
maps~\cite{Kazakov1986}. These are extremely interesting models from
the point of view of theoretical physics. In particular, the Ising
model on random maps is known to display a phase transition and a
critical point~\cite{Boulatov1987, Duits25}, at which the large-scale geometry
of maps is expected to be described by the so-called Liouville quantum
gravity metric at $\gamma=\sqrt{3}$~\cite{Ding2023}, i.e.\ central charge $c=1/2$. Confirming this expectation by rigorous mathematical results is a major open problem,
and gaining a better combinatorial understanding of
hypermaps will be helpful in such a task. Our focus is on hypermaps
with a ``mixed boundary'', as such objects arise when trying to apply
the method of peeling~\cite{Curien2023} to hypermaps.

Generating functions of these mixed boundary hypermaps were first studied 
in \cite{Eynard2002,Eynard2003} by turning Tutte's edge removal recursions 
into algebraic equations, and solving these.  
All these generating functions are algebraic functions on a genus zero 
spectral curve $E(x,y)=0$, whose rational parametrization $(x(z), y(z))$ is uniquely determined by equations \eqref{eq:defxynew} and \eqref{eq:aibieq} below.
Further work was carried out in \cite{EO2005,Eynard2016} in order to obtain
 explicit expressions using algebraic methods. It is natural to ask 
 if these expressions can be obtained through a \emph{combinatorial} approach, in order to 
 obtain a better understanding of hypermaps.
 This is the question we solve in this paper and the next one \cite{Lejeune2026}.

\paragraph{Acknowledgments.} We thank Marie Albenque and Emmanuel
Guitter for useful discussions. This work is supported by the ERC-SyG
project, Recursive and Exact New Quantum Theory (ReNewQuantum), which
received funding from the European Research Council (ERC) under the
European Union's Horizon 2020 research and innovation programme under
grant agreement No 810573, and by the Agence Nationale de la Recherche
via the grant ANR-23-CE48-0018 ``CartesEtPlus''.

\subsection{Basic definitions and notations}
\label{sec:hyperdef}

We start by introducing some basic definitions, and refer
to~\cite{Schaeffer2015} for more
background. A \emph{plane map}, hereafter called map for short, is a
connected finite graph drawn in the plane\footnote{The notion of plane
  map differs slightly from that of planar map, which is a graph drawn
  on the \emph{sphere}. A plane map corresponds to a planar map with a
  distinguished face, chosen as the outer face when projecting the
  sphere on the plane.} without edge crossings, and considered up to
homeomorphism. Loops and multiple edges are allowed.  A map consists
of \emph{vertices} and \emph{edges} stemming from the graph structure,
and of \emph{faces} which are the connected components of the
complement of the graph in the plane. Among these, the bounded ones
are called \emph{inner} faces, and the unbounded one is called the
\emph{outer} face. A \emph{corner} is the angular sector between two
edges appearing consecutively around a same vertex. It corresponds to
an incidence between a vertex and a face. The \emph{degree} of a vertex
or face is its number of corners.

A \emph{bridge} is an edge whose removal disconnects the map, or equivalently, since 
our map is plane, a bridge is an edge having the same face on its two sides. A \emph{path} is a sequence of 
consecutive edges; it is said to be \emph{closed} if it starts and ends at the same vertex. The \emph{contour} of a face is the closed path formed by its incident edges. We consider below maps that are oriented, i.e.\ where each edge 
has an intrinsic orientation. A path is said to be \emph{directed} if it respects the orientation of its edges.

A \emph{plane hypermap}\footnote{What we call a plane hypermap here
  corresponds, in the terminology
  of~\cite{Albenque2026}, to a planar hypermap
  with one non-monochromatic boundary, which is selected as the outer
  face.}, hereafter simply called a hypermap, is a plane map whose
edges are oriented in such a way that the contour of each inner face
forms a closed directed path. We do not require the contour of the outer
face to be directed.  An inner face is called \emph{white} if its
contour is directed clockwise, and \emph{black} if it is directed
counter-clockwise. Note that adjacent inner faces always have opposite
colours, and hence that an inner face cannot be incident to a bridge.

Given a hypermap and two corners $c,c'$ incident to the outer face,
the \emph{boundary interval} $[c,c']$ is the portion of the contour of the outer face that
lies between $c$ and $c'$, when turning counter-clockwise around the
hypermap.

\begin{figure}[t]
    \centering
    \begin{subfigure}[b]{0.75\textwidth}
        \includegraphics[width=\textwidth]{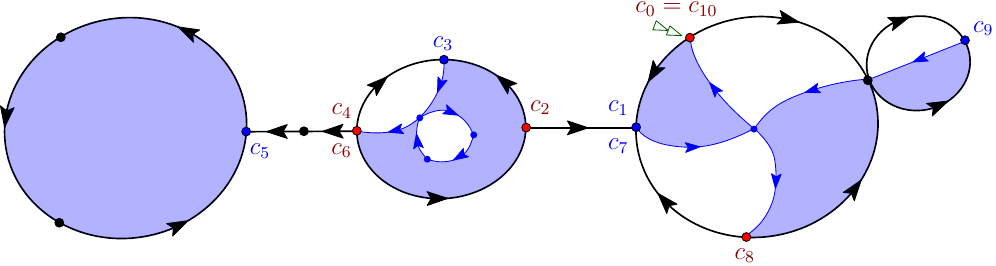}
        \caption{A $5$-alternating hypermap contributing $\displaystyle\frac{t^{16}\,(t_2^{\circ})^2\,(t_3^{\circ})^2\,t_2^{\bullet}\,(t_3^{\bullet})^{3}\,t_7^{\bullet}}{x_1^2\,y_1^2\,x_2^2\,y_2^2\,x_3^6\,y_3^3\,x_4^3\,y_4^2\,x_5^3\,y_5^3}$ to $H_5$.}
        \label{fig:hypermap_generale}
    \end{subfigure}
      \bigskip
    
    \begin{subfigure}[b]{0.42\textwidth}
        \includegraphics[width=\textwidth]{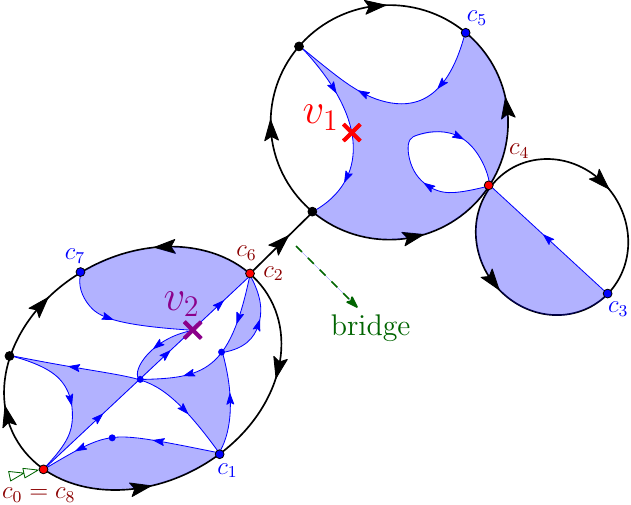}
        \caption{A $4$-alternating hypermap, accessibly pointed if $v_1$ is marked and not if $v_2$ is because of the presence of a bridge.}
        \label{fig:hypermap_accessible}
      \end{subfigure}
    \hfill
    \begin{subfigure}[b]{0.37\textwidth}
      \includegraphics[width=\textwidth]{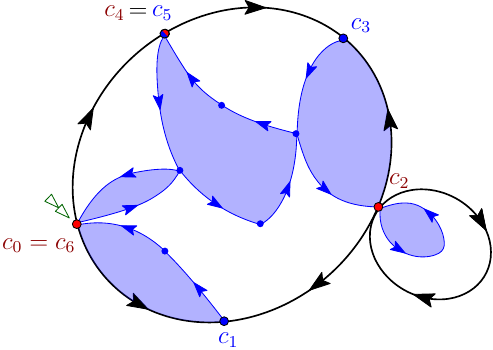}
        \caption{A strongly connected $3$-alternating hypermap contributing $\displaystyle\frac{t^{10}\,t_2^{\circ}\,t_3^{\circ}\,t_4^{\circ}\,t_7^{\circ}\,t_1^{\bullet}\,t_2^{\bullet}\,(t_3^{\bullet})^{2}\,t_5^{\bullet}}{x_1^2\,y_1^3\,x_2^2\,y_2^2\,x_3^1\,y_3^2}$}
        \label{fig:hypermap_blob}
      \end{subfigure}
    \caption{Examples of $k$-alternating hypermaps for different values of $k$. For readability, black faces are shown in blue and the boundary is shown in black. Also, marked corners with even (resp. odd) indices are shown in red (resp. blue).}
    \label{fig:hypermap}
\end{figure}

For $k$ a positive integer, a hypermap $\mathfrak{m}$ is said to be
\emph{$k$-alternating} if its outer face carries $2k$ marked incident
corners $c_0,c_1,c_2,\ldots,c_{2k-1},c_{2k}=c_0$, appearing in counter-clockwise
order and not necessarily distinct (as in Figure \ref{fig:hypermap_blob}), such that for all $i=1,\ldots,k$,
the boundary interval $[c_{2i-2},c_{2i-1}]$ is directed from $c_{2i-2}$ to $c_{2i-1}$ (i.e.\ counter-clockwise) and the boundary
interval $[c_{2i-1},c_{2i}]$ is directed from $c_{2i}$ to $c_{2i-1}$ (i.e.\ clockwise). See Figure~\ref{fig:hypermap} for some examples. We denote the lengths (number
of edges) of these boundary intervals by $\ell_i^\bullet(\mathfrak{m})$
and $\ell_i^\circ(\mathfrak{m})$ respectively.

Let $t,t_1^\circ,t_2^\circ,\ldots,t_1^\bullet,t_2^\bullet,\ldots$
be a collection of formal variables and $\mathcal{R}:=\Q[\![t,t_1^\circ,t_2^\circ,\ldots,t_1^\bullet,t_2^\bullet,\ldots]\!]$ the ring of formal power series in these variables. To a hypermap $\mathfrak{m}$, we
generally assign a weight
\begin{equation}
  \label{eq:wmdef}
  w(\mathfrak{m}) = t^{\#\{\text{vertices of } \mathfrak{m}\}} \prod_{\substack{f \text{ white}\\\text{inner face}}} t_{\deg(f)}^\circ \prod_{\substack{f \text{ black}\\\text{inner face}}} t_{\deg(f)}^\bullet\in\R,
\end{equation}
where $\deg$ stands for the degree. Given extra formal variables
$x_1,\ldots,x_k,y_1,\ldots,y_k$, we then set
\begin{equation}
  \label{eq:Hkdef}
  H_k(x_1,\ldots,x_k;y_1,\ldots,y_k; t, (t_i^{\circ})_{i\geqslant1},(t_j^{\bullet})_{j\geqslant1}) := \delta_{k,1} + \sum_{\mathfrak{m}} \frac{w(\mathfrak{m})}{
    x_1^{\ell_1^\bullet(\mathfrak{m})+1} y_1^{\ell_1^\circ(\mathfrak{m})+1} \cdots
    x_k^{\ell_k^\bullet(\mathfrak{m})+1} y_k^{\ell_k^\circ(\mathfrak{m})+1}},
\end{equation}
where the sum is over all $k$-alternating hypermaps
$\mathfrak{m}$. This quantity is a formal power series in $t$, the
$t_d^\circ$ and $t_d^\bullet$ for $d \geq 1$, and the $x_i^{-1}$ and
$y_i^{-1}$ for $i\in[\![1,k]\!]$, hence an element of $\big(\prod_{i=1}^k x_i^{-1}y_i^{-1}\big)\R[\![(x_i^{-1})_{1\leqslant i\leqslant k},(y_j^{-1})_{1\leqslant j\leqslant k}]\!]$. Using inverse variables here, and
adding the conventional term $\delta_{k,1}$, are choices made to be
consistent with~\cite{Eynard2016}. 

In the present paper we will give formulas for $H_1$ and $H_2$,
leaving the case of general $k$ to a subsequent
paper~\cite{Lejeune2026}. However, we will obtain a general formula
for a variant of $H_k$ involving so-called accessibly pointed
hypermaps.

More precisely, we say that a vertex $v$ of a hypermap is
\emph{accessible} if, for every other vertex $u$, there exists a
directed path going from $u$ to $v$. Note that we do not require the existence of a directed path from $v$ to $u$. An \emph{accessibly pointed}
hypermap is a hypermap endowed with a distinguished accessible vertex.
See Figure~\ref{fig:hypermap_accessible}. We
denote by $A_k(x_1,\ldots,x_k;y_1,\ldots,y_k)$ the generating function
of accessibly pointed $k$-alternating hypermaps, obtained by just
changing the summation set in the right-hand side of~\eqref{eq:Hkdef}
and removing the conventional term $\delta_{k,1}$.

Finally, we define similarly the generating function
$B_k(x_1,\ldots,x_k;y_1,\ldots,y_k)$ of strongly connected
$k$-alternating hypermaps. By \emph{strongly connected}, we mean that
for every pair $(u,v)$ of vertices, there exists a directed path from
$u$ to $v$; equivalently, this means that every vertex is accessible,
see Figure~\ref{fig:hypermap_blob}. As we will see in
Section~\ref{sec:prelimlem} below, a hypermap is strongly connected if
and only if it contains no bridge.

We are now ready to state the main results of this paper.

\subsection{Main results}

Our first result is a general formula for
$A_k(x_1,\ldots,x_k;y_1,\ldots,y_k)$ which, as in~\cite{Eynard2002, Albenque2026}, is expressed in terms 
of some auxiliary quantities that we introduce first. Let us consider the formal
Laurent power series
\begin{equation}\label{eq:defxynew}
  x(z):=\sum_{i\geq -1} a_i z^{-i}\quad\text{ and }\quad y(z):=\sum_{i \geq -1} b_i z^i
\end{equation}
where the $a_i$ and $b_i$ are the elements of $\R$
determined recursively by the conditions
\begin{equation}
  \label{eq:aibieq}
  \begin{split}
    a_i = t \delta_{i,-1} + \sum_{d \geq 1} t^\bullet_d [z^{-i}] y(z)^{d-1} \qquad (i \geq -1)\\
    b_{-1} = 1, \qquad  b_i = \sum_{d \geq 1} t^\circ_d [z^i] x(z)^{d-1} \qquad (i \geq 0).
  \end{split}
\end{equation}
Here, the notation $[z^i]$ means that we extract the coefficient of
$z^i$ in the Laurent power series on the right.

\begin{thm}
  \label{thm:ak} For any $k \geq 1$, the generating function of accessibly pointed $k$-alternating
  hypermaps reads
  \begin{equation}\label{eq:ak}
    \begin{split}
      A_k(x_1,\ldots,x_k;y_1,\ldots,y_k) &= \sum_{\ell_1^\bullet,\ell_1^\circ,\ldots,\ell_k^\bullet,\ell_k^\circ \geq 0} \frac{[z^0] x(z)^{\ell_1^\bullet+\cdots+\ell_k^\bullet} y(z)^{\ell_1^\circ+\cdots+\ell_k^\circ}}{x_1^{\ell_1^\bullet+1}y_1^{\ell_1^\circ+1}\cdots x_k^{\ell_k^\bullet+1}y_k^{\ell_k^\circ+1}}\\
      &=[z^0] \frac{1}{\prod_{i=1}^k (x_i-x(z))(y_i-y(z))}.
    \end{split}
  \end{equation}
\end{thm}

This theorem will be proved using slice decomposition in
Section~\ref{sec:access}. The symmetry of $A_k$ under permutation of
$x_1,\ldots,x_k,y_1,\ldots,y_k$ is remarkable and
unexpected. Indeed, from the combinatorial definition of $H_k$, $A_k$
and $B_k$, we only expect symmetries such as invariance under rotation
$(x_1,\ldots,x_k;y_1,\ldots,y_k) \mapsto
(x_2,\ldots,x_k,x_1;y_2,\ldots,y_k,y_1)$ or invariance under reflection
$(x_1,\ldots,x_k;\allowbreak y_1,\ldots,y_k) \mapsto
(x_k,x_{k-1},\ldots,x_2,x_1;y_{k-1},y_{k-2},\ldots,y_1,\allowbreak
y_k)$. This latter invariance corresponds to the weight-preserving involution on the set
of plane hypermaps which consists in performing a reflection of the
plane and reversing the orientation of the edges. Note that this
involution does not preserve the orientation of edges,
hence $A_k$ has a priori no reason to be invariant under such a
transformation.

Now, we would like to deduce expressions for $H_k$ and $B_k$ from that
of $A_k$ above. As said previously,
we only consider the cases $k=1,2$ in the present paper. Let us first consider the case $k=1$:

\begin{prop}
  \label{prop:k1}
  We have the relations
  \begin{equation}
    \label{eq:a1h1b1rel}
    A_1(x;y) = \frac{\partial}{\partial t} \ln H_1(x;y)
    =  \frac{\frac{\partial}{\partial t} B_1(x;y)}{1-B_1(x;y)}.
  \end{equation}
  As a consequence, we have
  \begin{equation}
    \label{eq:h1b1}
    H_1(x;y) = \frac{1}{1-B_1(x;y)} = \exp\left( [z^0]z\ln\left(1-\frac{y(z)}{y}\right)\frac{\partial}{\partial z}\ln\left(1-\frac{x(z)}{x}\right)
     \right).
  \end{equation}
\end{prop}
The expression~\eqref{eq:h1b1} was previously given
in~\cite{Albenque2026}\footnote{In this reference, $1$-alternating hypermaps
  are called hypermaps with a Dobrushin boundary.}, where it is also
explained how it matches with the results stated
in~\cite{Eynard2002,Eynard2016}. Here, we give a new proof of this expression,
which consists in integrating~\eqref{eq:a1h1b1rel} and using
Theorem~\ref{thm:ak} for $k=1$.

Moreover, with the aim of giving various formulations for the central case $k=1$, we will show the following identity thanks to Proposition \ref{prop:k1},:
\begin{prop}\label{prop:a1E}
Assume white and black faces have bounded degrees, which means that we set $t_i^{\circ}$ and $t_j^{\bullet}$ to $0$ for $i$ and $j$ large enough so that $x(z)$ and $y(z)$, defined in Equation \eqref{eq:defxynew}, are now Laurent polynomials.\\
Let $E(x,y)$ be the resultant with respect to $z$ of $(x(z)-x,y(z)-y)$. Also let $Y(x)$, resp. $X(y)$, be the unique Laurent power series in $x^{-1}$ such that $E(x,Y(x))=0$ (resp. the unique Laurent power series in $y^{-1}$ such that $E(X(y),y)=0$). Then we have the formula
\begin{equation}\label{eq:a1E}
A_1(x;y)=\frac{\partial}{\partial t}\ln\left(\frac{E(x,y)}{(x-X(y))(y-Y(x))}\right).
\end{equation}
\end{prop}
We will detail why $x(z)-x$ and $y(z)-y$ are Laurent polynomials, and no longer Laurent power series, in the 
case of bounded face degrees while proving this proposition in Section \ref{sec:a1E}.
\\

Finally, for the case $k=2$, we will establish the following:
\begin{thm}
  \label{thm:h2b2}
  The generating function of $2$-alternating hypermaps reads
  \begin{equation}\label{eq:h2}
    H_2(x_1,x_2;y_1,y_2) = \frac{H_1(x_1;y_1)H_1(x_2;y_2)-H_1(x_1;y_2)H_1(x_2;y_1)}{
      (x_1-x_2)(y_1-y_2)}
  \end{equation}
  and that of strongly connected $2$-alternating hypermaps reads
\begin{equation}\label{eq:b2}
  B_2(x_1,x_2;y_1,y_2) = \frac{H_1(x_1;y_2)^{-1}H_1(x_2;y_1)^{-1}-H_1(x_1;y_1)^{-1}H_1(x_2;y_2)^{-1}}{
    (x_1-x_2)(y_1-y_2)}.
\end{equation}
\end{thm}
The first expression corresponds
to~\cite[Corollary~8.4.1]{Eynard2002,Eynard2016}, and we here give a
combinatorial proof, as asked in this reference.

\subsection{Outline}

In Section \ref{sec:enumaccess}, we study the case of accessibly pointed hypermaps, where every vertex
can be connected to a distinguished vertex, after redefining the essential tool for their
enumeration: the hypermap slices. The realization of the link between accessible
hypermaps and slice decomposition is expressed in Theorem \ref{thm:ak}. 

We will then have the main tools required to prove relation
\eqref{eq:a1h1b1rel} concerning $1$-alternating hypermaps in Section \ref{sec:1alt_tranche}. We will thus
use these formulas to recover more classical results on this type of hypermaps, such
as Equation \eqref{eq:h1b1}, in Section \ref{sec:1alt_autre_relation}, and Proposition \ref{prop:a1E} in Section \ref{sec:a1E}. Section \ref{sec:1alt} as a whole will then establish Propositions \ref{prop:k1} and \ref{prop:a1E}. 

Finally, in Section \ref{sec:2alt}, we prove Theorem \ref{thm:h2b2}, which illustrates our general strategy: we will mark a vertex and then integrate 
in order to recover a formula in which $t$ does not appear explicitly. 

The case of $k$-alternating boundaries and the associated general formulas for $H_k$ and $B_k$ will appear in the
second part of the article \cite{Lejeune2026}.

\section{Bridges and accessibility in plane hypermaps}
\label{sec:prelimlem}

Let us record here a few elementary observations related to the notion
of accessibility and strong connectivity in plane hypermaps. We start
with a simple but key lemma.

\begin{lem}\label{lem:pathnobridge}
  Given a plane hypermap, let $u$ and $v$ be two vertices such that there exists a non-directed path from $u$ to $v$ that does not contain a bridge. Then, there exists a directed path from $u$
  to $v$ (and vice versa).
\end{lem}

\begin{proof}
  By the assumption, there exists a not necessarily directed path
  $P$ from $u$ to $v$ which does not pass through a bridge. We may
  then construct a directed path $P'$ from $u$ to $v$ as
  follows. Consider every edge $e$ that appears in the wrong direction
  along $P$. Since $e$ is not a bridge, it is by planarity incident to at least one
  inner face $f$. The contour of $f$, being a directed cycle, provides
  a way to replace $e$ by a directed path (of possibly more than one
  edge) going in the right direction.
\end{proof}

As a consequence of our key lemma we get a full characterization of
strong connectedness for plane hypermaps.

\begin{lem}\label{lem:stronglyconnected}
  A plane hypermap is strongly connected if and only if it contains no
  bridge.
\end{lem}
For instance, in Figure \ref{fig:hypermap}, only the third hypermap is strongly connected
because of the existence of respectively three and one bridge in the first and
second hypermaps.

\begin{proof}[Proof of Lemma~\ref{lem:stronglyconnected}]
  Let $\mathfrak{m}$ be a plane hypermap. If $\mathfrak{m}$ has no
  bridge then it is strongly connected by Lemma~\ref{lem:pathnobridge}.
  Conversely, if $\mathfrak{m}$ contains a bridge $b$, then there
  exists no directed path going from the endpoint of $b$ to its
  origin.
\end{proof}

\begin{rem}
Note that the proof of Lemma~\ref{lem:pathnobridge}, and hence Lemma~\ref{lem:stronglyconnected}, rely crucially on planarity and on the fact that there exists only one 
face whose contour is not a directed cycle. Figure \ref{fig:nobridgenostrong} displays 
two examples of oriented maps which are not strongly connected despite having no bridge.
\end{rem}

\begin{figure}
  \centering
  \includegraphics{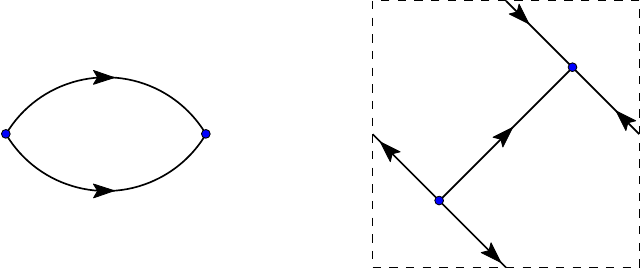}
  \caption{Two maps with oriented edges which are not strongly connected despite having no bridge. The left one is planar and has two faces, none of which forms a directed cycle. The right one has genus one (we identify opposite sides of the dashed square) and one face.}
  \label{fig:nobridgenostrong}
\end{figure}

\begin{rem}\label{rem:stronglyconnected}
Another way to understand this lemma is to define an equivalence 
relation between the vertices of a hypermap. For $u,v$ vertices of a fixed hypermap, one sets $u\leftrightarrow v$ if, 
and only if, there exists a directed path from $u$ to $v$ and a directed path from $v$ to $u$.

Then Lemma \ref{lem:pathnobridge} and Lemma \ref{lem:stronglyconnected} imply 
that the equivalence class of a vertex with respect to this equivalence relation is a strongly connected hypermap. Hence any alternating hypermap 
can be divided into strongly connected maps. This will be useful in Sections \ref{sec:a1b1} and \ref{sec:b2}.
\end{rem}

Another consequence of our key lemma is the following:

\begin{lem}\label{lem:accesscrit}
  For a vertex $v$ of a plane hypermap to be accessible, it is
  sufficient to ensure that it can be attained from any vertex
  incident to the outer face.
\end{lem}

\begin{proof}
  This follows from the fact that, for any vertex $u$ of a plane
  hypermap, we may construct a directed path going from $u$ to a
  vertex incident to the outer face. To see this, recall that bridges
  may only be incident to the outer face, and apply
  Lemma~\ref{lem:pathnobridge}.
\end{proof}

\begin{rem}\label{rem:inwardcorner}
  \begin{figure}
    \centering
    \begin{subfigure}[b]{0.31\textwidth}
        \includegraphics[width=\textwidth]{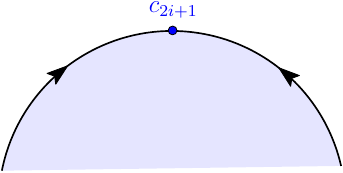}
        \caption{\centering An inward corner}
    \end{subfigure}
    \hspace{2cm}
    \begin{subfigure}[b]{0.31\textwidth}
        \includegraphics[width=\textwidth]{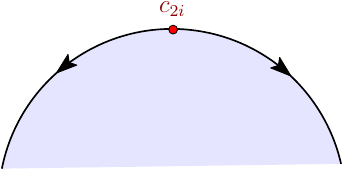}
        \caption{\centering An outward corner}
    \end{subfigure}
    \caption{Sketches of inward and outward corners. Inner faces are shown in light blue.}
    \label{fig:corner}
  \end{figure}

  It is even sufficient to ensure accessibility from
  the \textbf{inward corners}, that is the outer corners having two
  incoming outer edges, see Figure~\ref{fig:corner}. In an accessibly
  pointed $k$-alternating hypermap, the inward corners are
  marked corners with odd indices, while those
  having two outgoing outer edges are called the \textbf{outward
    corners} and have even indices. Note that the even or odd marked corners are not all necessarily inward corners or outward corners: this situation 
    occurs when some marked corners coincide, as illustrated on Figure \ref{fig:hypermap_blob} where $c_4=c_5$ is neither inward nor outward.
\end{rem}

\section{Enumeration of accessibly pointed $k$-alternating hypermaps}
\label{sec:enumaccess}

The main goal of this section is to prove Theorem~\ref{thm:ak}.  To
this end, we make use of the slice decomposition, introduced in the
context of hypermaps in~\cite{Albenque2026}.

\subsection{Slices}
\label{sec:slices}

\paragraph{Reminders.}

Let us start by recalling the definition of slices of types
$\mathcal{A}$ and $\mathcal{B}$ given
in~\cite{Albenque2026}:

\begin{Def} \label{Def:sliceAB}
  A \emph{slice of type} $\mathcal{A}$ or \emph{of type} $\mathcal{B}$
  is a plane hypermap with three distinguished corners, denoted $o$,
  $l$, and $r$, appearing in counterclockwise order around the outer
  face, and satisfying the following conditions:
  \begin{itemize}
  \item the boundary interval $[o,l]$, called the~\emph{left side},
    forms a directed path from $l$ to $o$ of minimal length, i.e.\ is a \emph{geodesic} from $l$ to $o$,
  \item the boundary interval $[r,o]$, called the~\emph{right side},
    forms the \textbf{unique} directed path from $r$ to $o$ of minimal
    length, i.e.\ is the unique geodesic from $r$ to $o$,
  \item the left and right sides only meet at the vertex incident to
    $o$, called the~\emph{apex},
  \item the boundary interval $[l,r]$, called the~\emph{base}, forms a
    directed path from $l$ to $r$ for a slice of type $\mathcal{A}$,
    and a directed path from $r$ to $l$ for a slice of type
    $\mathcal{B}$.
  \end{itemize}
  When the base has length one, the slice is said to be \emph{elementary}.
\end{Def}

For any integer $i$, let us denote by $a_i$ (resp.\ $b_i$) the
generating function of elementary slices of type $\mathcal{A}$ (resp.\
$\mathcal{B}$), in which the difference between the length of the
right side and the length of the left side is equal to $i$ (resp.\ $-i$). Here,
the weight of a slice is defined by a slight variant of the
formula~\eqref{eq:wmdef}, namely we do not attach a weight $t$ to the
vertices belonging to the right side. It is straightforward to check
that $a_i=b_i=0$ for $i < -1$ as there exists no slice contributing to
these generating functions, and it was shown
in~\cite{Albenque2026} that the sequences
$(a_i)_{i \geq -1}$ and $(b_i)_{i \geq -1}$ are determined recursively
by~\eqref{eq:defxynew} and~\eqref{eq:aibieq}.

\begin{rem}
  These recursive equations also appear
  in~\cite[Chapter~8]{Eynard2016}. Hypermaps
  are obtained by setting the variables denoted $a,b,c$ in this
  reference to $a=b=0$, $c=-1$. Our quantities $a_k$ and $b_k$ are
  related to $\gamma,\alpha_k,\beta_k$ of this reference by
  \begin{equation}
    \label{eq:Eyn_connection}
    a_{-1} = \gamma^2, \qquad a_k = \frac{\alpha_k}{\gamma^k}, \qquad
    b_k = \beta_k \gamma^k,
  \end{equation}
  which essentially corresponds to a change of variable
  $z \mapsto \gamma z$ in $x(z)$ and $y(z)$.
\end{rem}

\paragraph{Generalized slices.} For our purposes it is convenient to consider a slightly generalized notion of slices:

\begin{Def} \label{Def:genslice} A \emph{slice} is a plane hypermap
  with three distinguished corners, denoted $o$, $l$, and $r$,
  appearing in counterclockwise order on its outer face, and satisfying the
  first three items in Definition~\ref{Def:sliceAB}, the fourth
  item being replaced by the mere requirement that the apex be
  accessible.
\end{Def}

\begin{rem}
  Lemma~\ref{lem:accesscrit} ensures that a slice of type
  $\mathcal{A}$ or $\mathcal{B}$ is still a slice according to the
  above generalized definition, since the contour of the outer face
  consists of two directed paths ending at the apex. In a generalized
  slice, the accessibility requirement forbids the base $[l,r]$ from
  containing a bridge that points away from the apex.
\end{rem}

\begin{figure}[h!]
    \centering
    \begin{subfigure}[b]{0.31\textwidth}
        \includegraphics[width=\textwidth]{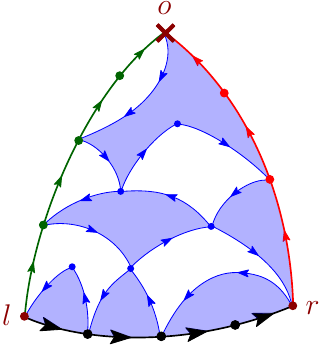}
        \caption{\centering A slice of type $\mathcal{A}$ with\\ base word $a^4$}
        \label{fig:tranche1}
    \end{subfigure}
    \hfill
    \begin{subfigure}[b]{0.31\textwidth}
        \includegraphics[width=\textwidth]{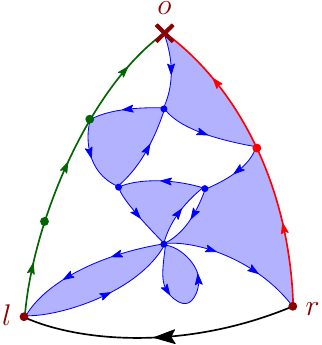}
        \caption{\centering An elementary slice of type $\mathcal{B}$ with base word $b$}
        \label{fig:tranche2}
    \end{subfigure}
    \hfill
    \begin{subfigure}[b]{0.29\textwidth}
        \includegraphics[width=\textwidth]{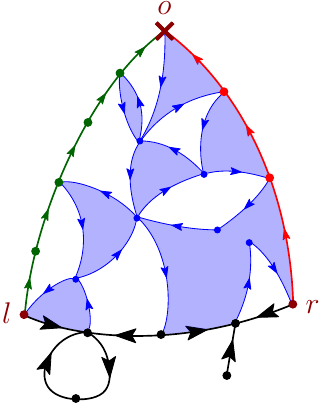}
        \caption{\centering A (generalized) slice with base word $ab^3abab$}
        \label{fig:tranche3}
    \end{subfigure}
    \caption{Examples of slices of different types and their associated word in $\{a,b\}$. Here the left side is represented in green while the right side is red and the base is black, as in \cite{Albenque2026}.}
    \label{fig:tranche}
\end{figure}

To each (generalized) slice, we associate its \emph{base word}, which
is a word in the alphabet $\{a,b\}$, as follows.  We consider the
orientation of the arrows along the base, read from $l$ to $r$: to
each arrow oriented from $l$ to $r$, we associate the letter $a$, and
to each arrow oriented from $r$ to $l$, the letter $b$. See
Figure~\ref{fig:tranche} for examples. Note that a slice has type
$\mathcal{A}$ (resp.\ $\mathcal{B}$) if and only if its base word
contains only $a$'s (resp.\ $b$'s). In an elementary slice the base
word has a single letter.

Interestingly, the generating function of slices with a prescribed
base word can be straightforwardly expressed in terms of the
quantities $x(z)$ and $y(z)$ counting elementary slices. 

\begin{prop} \label{prop:slicewordgf}
  Let $w$ be a word in the alphabet $\{a,b\}$, and let $i$ be an integer. Then, the generating function of slices with base word $w$, where
  the difference between the length of the left side and that of the right side is equal to $i$, is given by
  \begin{equation}
    [z^i] x(z)^{n_a(w)} y(z)^{n_b(w)}
  \end{equation}
  where $n_a(w)$ (resp.\ $n_b(w)$) denotes the number of occurrences
  of the letter $a$ (resp.\ $b$) in $w$.
\end{prop}

\begin{proof}
  The proof is a direct generalization of that
  of~\cite[Proposition~2.8 and
    Corollary~2.9]{Albenque2026}: a slice can be bijectively decomposed
  into a sequence of elementary slices by cutting along the leftmost
  geodesics going from each outer corner along the base to the
  apex. Note that such leftmost geodesics always exist by the
  requirement that the apex be accessible. The new feature here is
  that we may mix elementary slices of types $\mathcal{A}$ and
  $\mathcal{B}$ in the sequence.
\end{proof}

\subsection{Bijection between slices and accessibly pointed hypermaps}
\label{sec:access}

Recall from Section~\ref{sec:hyperdef} the definition of accessibly pointed $k$-alternating hypermaps.

\begin{prop} \label{prop:accesbij}
  For any positive integer $k$ and any nonnegative integers
  $\ell_1^\bullet,\ell_1^\circ,\ldots,\ell_k^\bullet,\ell_k^\circ$,
  there is a weight-preserving bijection between the set of accessibly
  pointed $k$-alternating hypermaps $\mathfrak{m}$ such that
  $\ell_i^\bullet(\mathfrak{m})=\ell_i^\bullet$ and
  $\ell_i^\circ(\mathfrak{m})=\ell_i^\circ$ for all $i=1,\ldots,k$,
  and the set of slices with base word
  $a^{\ell_1^{\bullet}}b^{\ell_1^{\circ}}\cdots
  a^{\ell_k^{\bullet}}b^{\ell_k^{\circ}}$ where the left and right
  sides have the same length.
\end{prop}

\begin{proof}
  \begin{figure}[h!]
    \includegraphics[width=\textwidth]{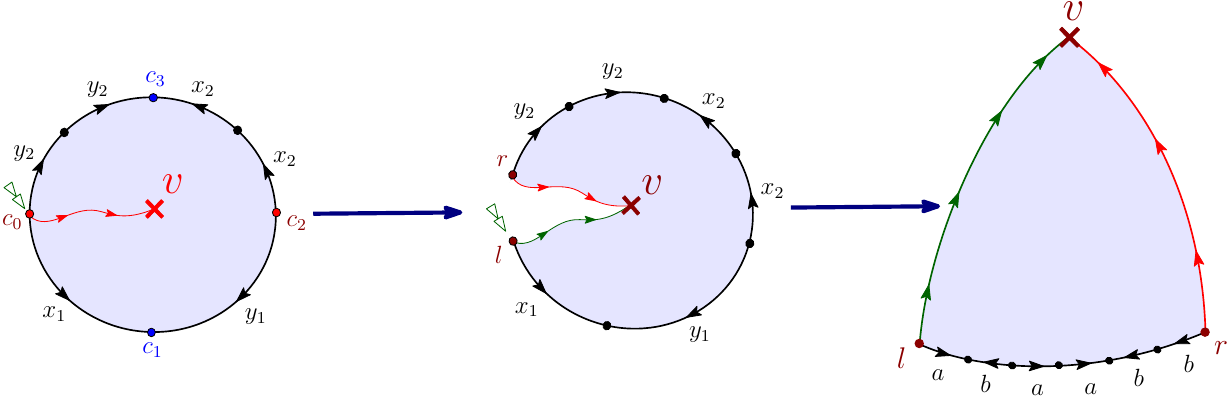}
    \caption{Sketch of the proof of Proposition \ref{prop:accesbij} in the case of a pointed $2$-alternating hypermap with $(\ell_1^{\circ},\ell_1^{\bullet},\ell_2^{\circ},\ell_2^{\bullet})=(1,1,2,2)$. The first step consists in taking the leftmost geodesic from the marked corner $c_0$ (left figure) to the marked vertex $v$. Then one cuts 
    along this geodesic, $c_0$ is now decomposed into to vertices: $r$ and $l$ (second figure), where $r$  and $l$ is connected to $v$ by resp. a unique 
    geodesic and a generic geodesic. Then a slice of base word $aba^2b^2$ rises by deforming continuously the second figure, which gives the third figure. Note that the marked vertex is now the apex.}
    \label{fig:decomp_en_tranche}
  \end{figure}
  
  The proof is illustrated on Figure~\ref{fig:decomp_en_tranche} and
  is similar to that
  of~\cite[Proposition~3.4]{Albenque2026}:
  starting from an accessibly pointed $k$-alternating hypermap with
  distinguished vertex $v$, we cut along the leftmost geodesic going
  from the first marked corner $c_0$ to $v$. Such a geodesic always
  exists since $v$ is assumed accessible. We may check that the
  resulting map is a slice, $v$ becoming the apex which remains
  accessible. Furthermore, the length constraints on the boundary
  intervals of the initial $k$-alternating hypermap entail that the
  base word of the resulting slice is
  $a^{\ell_1^{\bullet}}b^{\ell_1^{\circ}}\cdots
  a^{\ell_k^{\bullet}}b^{\ell_k^{\circ}}$. Finally, this construction is bijective as 
  the inverse bijection consists in gluing the left and right sides of the slice.
\end{proof}

By combining Propositions~\ref{prop:slicewordgf} and \ref{prop:accesbij} we obtain the following:
\begin{prop}\label{cor:word}
  The generating function of accessibly pointed $k$-alternating
  hypermaps $\mathfrak{m}$ such that
  $\ell_i^\bullet(\mathfrak{m})=\ell_i^\bullet$ and
  $\ell_i^\circ(\mathfrak{m})=\ell_i^\circ$ for all $i=1,\ldots,k$ is
  equal to
  \begin{equation}
    [z^0] x(z)^{\ell_1^\bullet+\cdots+\ell_k^\bullet} y(z)^{\ell_1^\circ+\cdots+\ell_k^\circ}.
  \end{equation}
\end{prop}

Theorem~\ref{thm:ak} follows by summing over
$\ell_1^\bullet,\ell_1^\circ,\ldots,\ell_k^\bullet,\ell_k^\circ$. Let
us conclude this section by recording the following useful expression for
$A_k$ in terms of $A_1$:

\begin{cor}\label{cor:ak}
 For any $k \geq 1$ we have
\begin{equation}\label{eq:A=decompo}
    A_k(x_1,\ldots,x_k;y_1,\ldots,y_k)
    = \sum_{1 \leqslant i,j \leqslant k}
    \frac{A_1(x_i; y_j)}
    {\displaystyle \prod_{p \neq i}(x_i - x_p)
     \prod_{q \neq j}(y_j - y_q)}.
 \end{equation}
 In particular for $k=2$ we have
 \begin{equation}\label{eq:A=decompo2}
   A_2(x_1,x_2;y_1,y_2) = \frac{A_1(x_1;y_1)+A_1(x_2;y_2)-A_1(x_1;y_2)-A_1(x_2;y_1)}{(x_1-x_2)(y_1-y_2)}.
 \end{equation}
\end{cor}

\begin{proof}
Equation~\eqref{eq:A=decompo} follows directly from equation~\eqref{eq:ak} by performing two partial fraction decompositions with respect to the variables $x(z)$ and $y(z)$, and then applying the case $k=1$ to recover $A_1(x_i; y_j)$.
\end{proof}

Note that the first formula of Corollary~\ref{cor:ak}, made explicit in the case of accessibly pointed $2$-alternating hypermaps, plays a key role when studying $k$-alternating hypermaps, while Theorem \ref{thm:ak} itself contains the full combinatorial interpretation of 
alternating hypermaps as slices. Both formulas are therefore essential, each in its own way.

\section{Enumeration of $1$-alternating hypermaps}
\label{sec:1alt}
In this section, we shall successively detail all the equalities appearing in Proposition \ref{prop:k1}.
We begin with the combinatorial approach based on slices, in order to prove the first equation in \eqref{eq:a1h1b1rel}.
We then show how to derive an explicit expression for $H_1$, as given in Equation \eqref{eq:h1b1}, 
before returning to the case of bounded-face degrees and establishing the connection with the spectral curve.

\subsection{How to use accessibility in pointed $1$-alternating hypermaps}
\label{sec:1alt_tranche}

The purpose of this section is to establish Equation \eqref{eq:a1h1b1rel} 
in Proposition \ref{prop:k1}. 

\subsubsection{Generic case} \label{sec:h1a1}

Let us start by proving the first equality in~\eqref{eq:a1h1b1rel} which amounts to 
\begin{equation}\label{eq:dH1/dt} \frac{\partial H_1(x;y)}{\partial t}=A_1(x;y)H_1(x;y),\end{equation}
   where we recall that $H_1$ is the generating function of $1$-alternating hypermaps
 and $A_1$ that of accessibly pointed $1$-alternating hypermaps. 
\\ 

We start by the observation that $\frac{\partial H_1}{\partial t}$ counts pointed, but not necessarily 
accessibly pointed, $1$-alternating hypermaps. We shall therefore discuss what could prevent, in a pointed 
$1$-alternating hypermap $\mathfrak{m}$, the marked vertex $v$ from being accessible. By Remark \ref{rem:inwardcorner}, 
$v$ is accessible if and only if it is reachable from $c_1$. 
Now, by Lemma \ref{lem:accesscrit}, we see that the obstruction to accessibility is the existence of a bridge separating $v$ and the vertex 
incident to $c_1$. See Figure \ref{fig:k1_notations} for an illustration of this situation and Figure \ref{fig:k1_preuve} for a sketch of the proof.

\begin{figure}[h!]
\centering 
\includegraphics[width=1\textwidth]{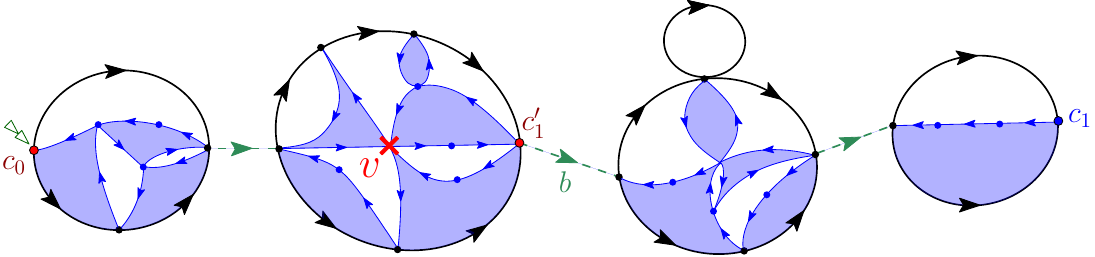} 
\caption{Sketch of a $1$-alternating hypermap. One can observe that the marked vertex $v$ is accessible from $c_0$ but not from $c_1$ because of the presence of bridges, dashed in green for readability, oriented toward $c_1$.}
\label{fig:k1_notations} 
\end{figure} 

This suggests the following decomposition of $\mathfrak{m}$ when the obstruction occurs. Among all 
the bridges which separate $v$ from $c_1$, let us denote by $b$ the one closest to $v$. 
Removing $b$ splits $\mathfrak{m}$ into two components: let us denote by $\mathfrak{m}_1$ the component 
containing $v$ and $\mathfrak{m}_2$ the one containing $c_1$. Note that $\mathfrak{m}_1$ and $\mathfrak{m}_2$ are 
both $1$-alternating hypermaps. Furthermore, the inward corner $c_1'$ of $\mathfrak{m}_1$ 
(at the position of the former bridge $b$) is by construction 
not separated from $v$ by a bridge. Hence, by Lemma \ref{lem:accesscrit} and Remark \ref{rem:inwardcorner}, $v$ is accessible 
in $\mathfrak{m}_1$, which makes $\mathfrak{m}_1$ accessibly pointed.
\\

We extend the decomposition of $\mathfrak{m}$ to the case where the obstruction does not occur by 
setting $\mathfrak{m}_1:=\mathfrak{m}$ and $\mathfrak{m}_2:=\emptyset$. This construction defines a map
$\mathfrak{m}\longmapsto(\mathfrak{m}_1,\mathfrak{m}_2)$ from the set of pointed $1$-alternating hypermaps 
to the set of pairs consisting of an accessibly pointed $1$-alternating hypermap and of a (possibly empty) 
unpointed $1$-alternating hypermap. It is clear that the mapping is bijective: when $\mathfrak{m}_2$ is not 
empty, the inverse mapping consists in connecting the inward corner of $\mathfrak{m}_1$ to the 
outward corner of $\mathfrak{m}_2$ by a bridge directed from the former 
to the latter.\\
Translating this bijection into the language of generating functions, we
get precisely Relation \eqref{eq:dH1/dt}, upon noting that the
case $\mathfrak{m}_2=\emptyset$ corresponds to the conventional term
$\delta_{k,1}$ present in the definition \eqref{eq:Hkdef} of $H_k$
(but absent in $A_1$), and that if $\mathfrak{m}_2 \neq \emptyset$
then we have
\begin{equation}
  \ell_1^\bullet(\mathfrak{m}) = \ell_1^\bullet(\mathfrak{m}_1) + \ell_1^\bullet(\mathfrak{m}_2) + 1, \qquad
  \ell_1^\circ(\mathfrak{m}) = \ell_1^\circ(\mathfrak{m}_1) + \ell_1^\circ(\mathfrak{m}_2) + 1
\end{equation}
which ensures that the exponents of $x$ and $y$ match.

  \begin{figure}[t]
    \centering
    \includegraphics[width=0.9\textwidth]{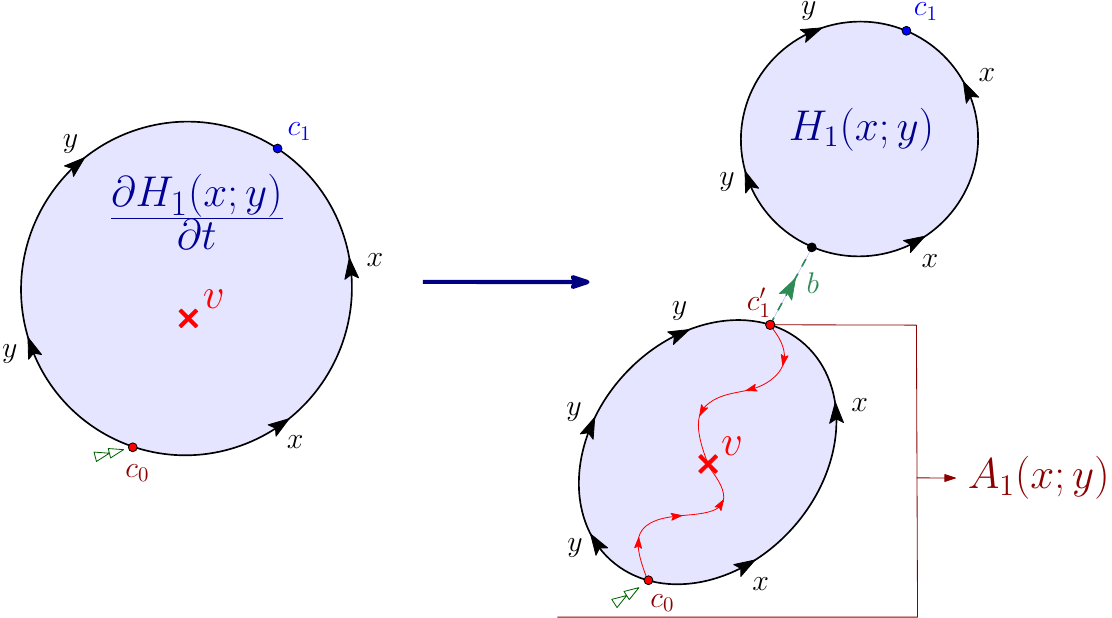} 
    \caption{Sketch of the proof for the case $k=1$, inner faces are blurred in light blue here. On the left, a generic sketch of a pointed $1$-alternating hypermap. On the right, the last bridge $b$ between $c_1$ and $v$ is dashed in green, and so $v$ is accessible from the base vertex of $b$, which leads to the desired decomposition.}
    \label{fig:k1_preuve}
\end{figure}

\begin{rem}
    Since $A_1$ has an explicit formula as a function of $x(z)$ and $y(z)$ derived from Theorem \ref{thm:ak}:
    $$ A_1(x;y)=[z^0]\frac{1}{(x-x(z))(y-y(z))}, $$ 
    we can finally use Equation \eqref{eq:dH1/dt} to write the following relation for $H_1$:
    \begin{equation}\label{eq:h1expa1}
    H_1(x;y)=\exp\left([z^0]\int_0^t \frac{1}{x-x(z,\tau)}\frac{1}{y-y(z,\tau)}\,\mathrm{d}\tau\right).
    \end{equation}
    The removal of the integral will be the main purpose of Section \ref{sec:1alt_autre_relation}.
\end{rem}

\subsubsection{Strongly connected case}\label{sec:a1b1}

Let us now establish the second equality of Equation \eqref{eq:a1h1b1rel} in Proposition \ref{prop:k1}:
\begin{equation}\label{eq:a1b1}
A_1(x;y)=\frac{\frac{\partial}{\partial t} B_1(x,y)}{1-B_1(x,y)},
\end{equation}
where $B_1$ is the generating function of strongly connected $1$-alternating hypermaps, which means maps with no bridge, in which every vertex is accessible by Lemma \ref{lem:stronglyconnected}.
\\

In fact, this relation is quite straightforward using Remark \ref{rem:stronglyconnected}, and is illustrated in Figure \ref{fig:a1b1}. 
Indeed, consider an accessibly pointed $1$-alternating hypermap $\mathfrak{m}$, and denote by $v$ its marked vertex. 
Then the set of all vertices 
accessible from $v$, i.e.\ the equivalence class of $v$ with respect to the relation $\leftrightarrow$, is a strongly 
connected $1$-alternating hypermap. We denote it by $\mathfrak{m}_1$. Note that $c_1$ belongs to $\mathfrak{m}_1$, 
otherwise there would exist a bridge in the wrong direction between $v$ and $c_1$. Since $v$ is accessible, this bridge should be oriented toward $v$, which leads to the existence of another inward corner that differs from $c_1$, and this is not possible since the hypermap is $1$-alternating.
\\
Let us now consider $c_0$. If $c_0$ is accessible from $v$, then it belongs to $\mathfrak{m}_1$, hence $v$ is accessible from any vertex and can access $c_0$ and $c_1$, so it can be related to any other vertex of the hypermap by Lemma \ref{lem:accesscrit} and Remark \ref{rem:inwardcorner}, so $\mathfrak{m}=\mathfrak{m}_1$. 
We then set $\mathfrak{m}_2:=\emptyset$ in this situation. However, if $c_0$ is not accessible from $v$, then there should exist multiple bridges, oriented from $c_0$ toward $v$, that obstruct $v$ from reaching $c_0$. 
By Lemma \ref{lem:stronglyconnected}, the component of $\mathfrak{m}$ that is between two consecutive bridges has to be a strongly connected $1$-alternating hypermap. We denote by $\mathfrak{m}_2$ the whole component of $\mathfrak{m}$ which contains all the vertices that are not accessible from $v$, which 
is a sequence of strongly connected hypermaps. 
\\
Then the map $\mathfrak{m}\longmapsto(\mathfrak{m}_1,\mathfrak{m}_2)$ from the set of accessibly pointed $1$-alternating hypermaps to couples of pointed strongly connected $1$-alternating 
hypermaps and a sequence of ordered unpointed strongly connected $1$-alternating hypermaps is bijective. The inverse map is defined as in the last section, namely by putting a bridge relating the outward corner of one strongly connected hypermap to the inward corner of the following strongly connected hypermap (assuming $\mathfrak{m}_1$ is at the end of this sequence).
Finally, using our convention on the exponents of $x$ and $y$ in alternating hypermaps, the language 
of generating functions translates this bijection into Relation \eqref{eq:a1b1}.

\begin{figure}[t]
    \centering
    \includegraphics[width=1\textwidth]{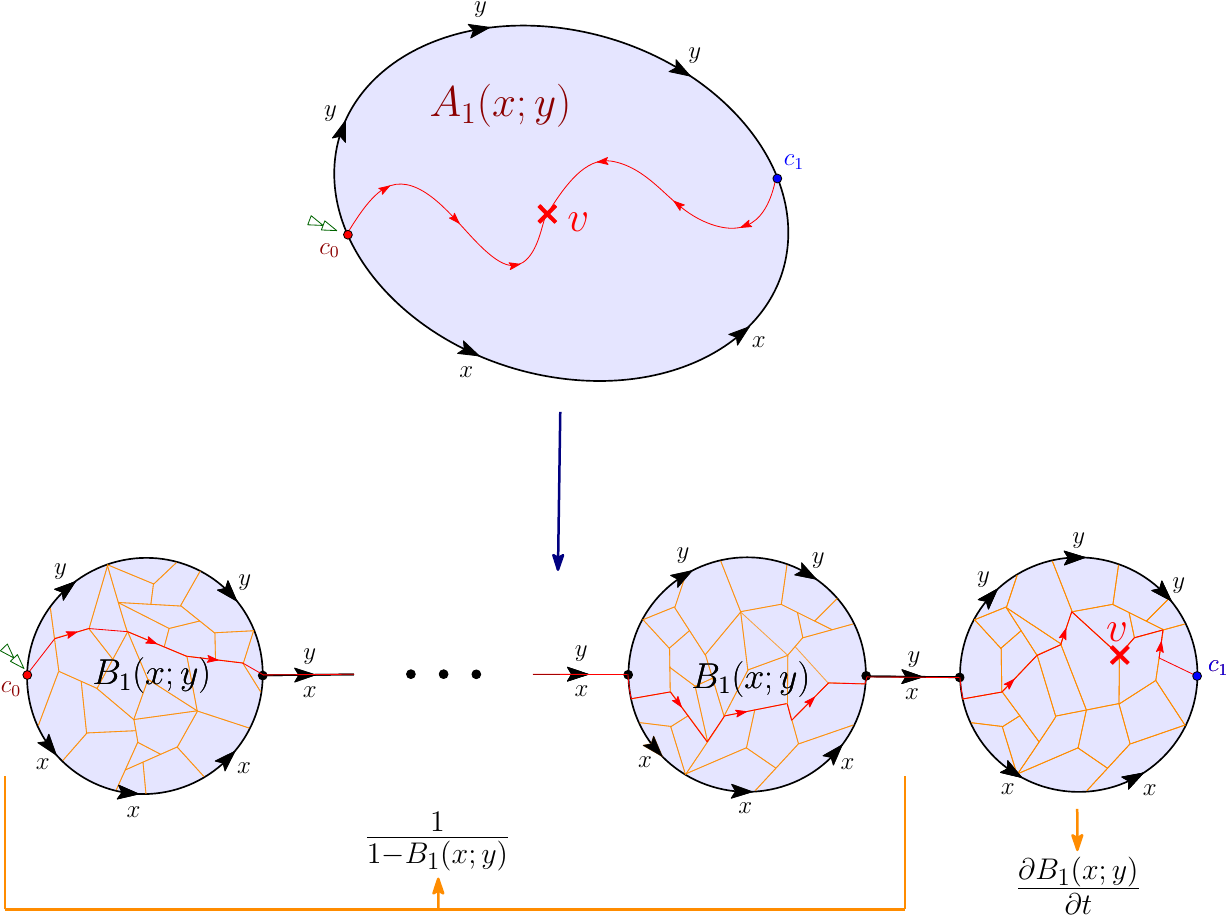} 
    \caption{Sketch of the proof of Equation \eqref{eq:a1b1}}
    \label{fig:a1b1}
\end{figure}

\begin{rem}
The present proof deals with the $1$-alternating case, which only allows bridges oriented in a single direction: from the outward corner $c_0$ to the inward corner $c_1$.  
We shall see in Section \ref{sec:b2} that discussing more general cases in this regard will be more challenging.
\end{rem}

Equations \eqref{eq:dH1/dt} and \eqref{eq:a1b1}, when combined and integrated, yield the first equality of Equation \eqref{eq:h1b1}:
\begin{equation}
H_1(x;y)=\frac{1}{1-B_1(x;y)}.
\end{equation}
A direct proof of this formula can be found in \cite{Albenque2026}, it is very similar to the above 
argument except that it does not involve a marked vertex.

\subsection{Various way to integrate relations for $1$-alternating hypermaps}
\label{sec:1alt_autre_relation}

In the previous section, we obtained various expressions for $H_1$ using the notion of accessibility.  
In particular, we made explicit Relation \eqref{eq:h1expa1}, which will now allow us, in this section, to derive new relations for $H_1$ with the aim of eliminating the integral with respect to the variable counting the number of vertices.  

To this end, we shall introduce a simpler case of hypermaps than those previously considered, namely hypermaps with a monochromatic boundary.  
We start by studying this type of hypermaps, showing how they are strongly related to the slices defined earlier.  
This will enable us to prove a Poisson formula for $x(z,t)$ and $y(z,t)$ in Lemma \ref{lem:poisson}, before providing proofs for Equations \eqref{eq:h1b1} and \eqref{eq:a1E}, which remain to be established.

\subsubsection{Review of the monochromatic boundary case}\label{sec:monochromatic}

In order to present a new relation between generating function of slices and of $H_1$ without any integral on $t$, we need to introduce a few notations by recalling what has already been done for the monochromatic boundary case.

\begin{Def}\label{def:monochromatic}
  Let $\mathfrak{M}^{\circ}$ (resp. $\mathfrak{M}^{\bullet}$) be the 
  set of $1$-alternating hypermaps $\mathfrak{m}$ such that $c_0$ and $c_1$ coincide and every 
  edge on the boundary is oriented counterclockwise (resp. clockwise). This is equivalent to saying that $\ell^{\circ}(\mathfrak{m})=0$ (resp. $\ell^{\bullet}(\mathfrak{m})=0$). 
  Each hypermap belonging to one of these two sets is said to have a \emph{monochromatic boundary}, respectively \emph{white} and \emph{black}, see Figure \ref{fig:monochromatic} for an illustration.
  \\
  One defines the corresponding generating functions as
  \begin{equation}
  W^{\circ}(x):=\sum_{\mathfrak{m}\in\mathfrak{M}^{\circ}}\frac{w(\mathfrak{m})}{x^{\ell^{\bullet}(\mathfrak{m})+1}}\in x^{-1}\R[\![x^{-1}]\!]\,\,\text{and}\,\,  W^{\bullet}(y):=\sum_{\mathfrak{m}\in\mathfrak{M}^{\bullet}}\frac{w(\mathfrak{m})}{y^{\ell^{\circ}(\mathfrak{m})+1}}\in y^{-1}\R[\![y^{-1}]\!]
  \end{equation}
  which are respectively exactly the coefficients $[y^{-1}]H_1(x;y)$ and $[x^{-1}]H_1(x;y)$ by definition. The weights $w$ are defined in Equation \eqref{eq:wmdef}.
\end{Def}

\begin{rem}
This definition is equivalent to taking a planar hypermap with one white (resp. black) marked face as boundary in order to define a hypermap with a white (resp. black) boundary.
\end{rem}

\begin{figure}[h!]
    \centering
    \begin{subfigure}{0.5\textwidth}
        \includegraphics[width=\textwidth]{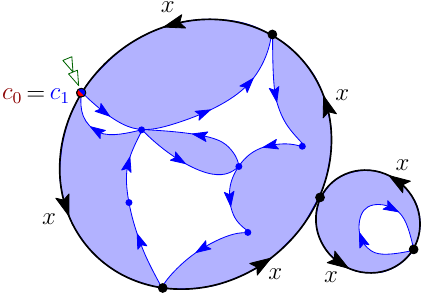}
        \caption{\centering a hypermap with white monochromatic boundary contributing to $W^{\circ}(x)$.}
    \end{subfigure}
    \begin{subfigure}{0.45\textwidth}
        \includegraphics[width=\textwidth]{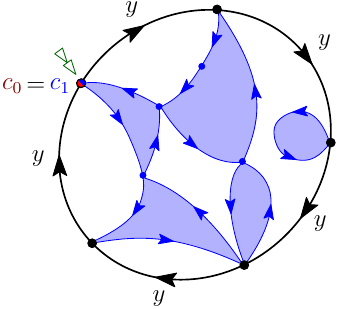}
        \caption{\centering a hypermap with black monochromatic boundary contributing to $W^{\bullet}(y)$.}
    \end{subfigure}
    \caption{Examples of hypermaps with monochromatic boundaries. Note that both marked corners are at the same position and that all faces incident to the boundary have the same color, hence the denomination "monochromatic".}
  \label{fig:monochromatic}
\end{figure}

Our first equation on these generating functions can shortly be found by a slice decomposition. 

\begin{prop}[{\cite[Corollary 3.5]{Albenque2026}}]\label{prop:dW/dt}
The generating functions of pointed hypermaps with monochromatic boundaries can be expressed as follows:
\begin{equation}\label{eq:dW/dt}
    \dpart{W^{\circ}}{t}(x,t)=[z^0]\frac{1}{x-x(z,t)}\,\,\,\text{and}\,\,\,\dpart{W^{\bullet}}{t}(y,t)=[z^0]\frac{1}{y-y(z,t)}.
\end{equation}

\end{prop}

The proof of this proposition is based on slice decomposition 
and one can sketch it as follows: since we are 
in the plane case and the boundary is monochromatic, 
no bridge can appear, and so each pointed hypermap with a monochromatic boundary is accessibly pointed. Hence they can be decomposed as a sequence of slices of type $\mathcal{A}$ or $\mathcal{B}$ and the result follows from Proposition \ref{cor:word}.
\\
In order to integrate Equation \eqref{eq:dW/dt} with respect to $t$, one could think of performing a partial fraction decomposition of the denominator of the right-hand side. 
To do so, we surely need to find the zeros of the denominator, which motivates the following result.

\begin{prop}[{{see also \cite[Section 5]{Albenque2026}}}]\label{defth:z0}
The equations $x(z)=x$ (resp. $y(z)=y$) admits exactly one solution $z^{\circ}(x)$ (resp. $z^{\bullet}(y)$) such that $z^{\circ}(x)^{-1}$ belongs to $x^{-1}\R[\![x^{-1}]\!]$ (resp. $z^{\bullet}(y)$ belongs to $y^{-1}\R[\![y^{-1}]\!]$). In fact, $z^{\circ}(x)$ (resp. $z^{\bullet}(y)$) also verifies the equation $z^{\circ}(x(z))=z$ (resp. $z^{\bullet}(y(z))=z$).\\
Furthermore, in the case of \emph{bounded face degrees}, the Laurent power series $Y(x)$ and $X(y)$ defined in Proposition \ref{prop:a1E} are given by
\begin{equation}\label{eq:XyYx}
Y(x)=y(z^{\circ}(x))\,\,\,\text{and}\,\,\,X(y)=x(z^{\bullet}(y)).
\end{equation}
\end{prop}

\begin{proof}
$\hookrightarrow$ We will only discuss the case of $y(z)=y$; the same discussion can 
be done for $x(z)=x$ by solving $x(z^{-1})=x$. Recall that 
$y(z)\in z^{-1}\R[\![z]\!]$ so one can write 
the equation as $z=\frac{\phi(z)}{y}$ where 
$\phi(z):=zy(z)\in\R[\![z]\!]$, such that 
$[z^0]\phi(z)=b_{-1}=1\neq0$. Hence the Lagrange's inversion
 theorem gives the existence and uniqueness of a compositional inverse $z^{\bullet}(y)\in y^{-1}\R[\![y^{-1}]\!]$ of $y(z)$. Let us emphasize that we do not need to extend $\R$ to a field in order to enable $z^{\bullet}(y)$ to have coefficients in $\R$.
\\
$\hookrightarrow$ The equations \eqref{eq:XyYx} come from the unicity of $Y(x)$ and $X(y)$ as Laurent power series in $x^{-1}$ and $y^{-1}$ respectively that satisfy $E(x,Y(x))=0$ and $E(X(y),y)=0$. Using the characterization of $E$ as a resultant, 
meaning such that $E(x(z),y(z))=0$, one only needs to take $z=z^{\circ}(x)$ (resp. $z=z^{\bullet}(y)$) and check that the corresponding functions $y(z^{\circ}(x))$ and $x(z^{\bullet}(y))$ are Laurent power series in $x^{-1}$ and $y^{-1}$ respectively to conclude.
\end{proof}

We are now ready to find an algebraic expression for $W^{\circ}(x)$ and $W^{\bullet}(y)$.

\begin{prop}[\cite{Eynard2002,Eynard2003,Albenque2026}]\label{prop:monochromatic}
    The notations of the previous definitions lead to the following expressions for the generating functions of monochromatic boundary hypermaps:
    \begin{equation}\label{eq:monochromatic}
        W^{\circ}(x)=Y(x)-V_{\circ}'(x)\,\,\,\text{and}\,\,\,W^{\bullet}(y)=X(y)-V_{\bullet}'(y),
    \end{equation}
    where $V_{\circ}$ and $V_{\bullet}$ are defined as
    \begin{equation}
         V_{\circ}(x)=\sum_{i\geqslant1}t_i^{\circ}\frac{x^i}{i}\,\,\,\text{and}\,\,\,V_{\bullet}(y)=\sum_{i\geqslant1}t_i^{\bullet}\frac{y^i}{i}.
    \end{equation}
\end{prop}

We refer the reader to \cite{Eynard2016} for a computational proof of this formula and to \cite{Albenque2026} for a combinatorial proof.
 Here we will give a third proof based on the formula stated in Proposition \ref{prop:dW/dt} and the notations of Section \ref{sec:a1E} in Appendix \ref{appendix}.

\begin{rem}
Historically, the pair $(x(z),y(z))$ was introduced in \cite{Eynard2002}
as a parametrization of the spectral curve $E(x,y)=0$, where $E(x,y)\in\R[x,y]$ naturally appears 
when solving loop equations. In fact, one of the most surprising results of \cite{Albenque2026} was the 
combinatorial interpretation of this curve by proving that $x(z)$ and $y(z)$ are also the generating functions of elementary slices. 
So we already have a lot more information on slices through the algebraic geometry perspective, which is 
another argument to state slices as a good cornerstone of the hypermap theory.
\end{rem}

In the spirit of the last proposition, another relation enables us to relate the generating functions of hypermaps 
with monochromatic boundaries and the specific roots $z^{\circ}$ and $z^{\bullet}$.
\begin{prop}[\cite{Albenque2026}]\label{prop:dmonochromatic}
The following derivative relations hold:
\begin{equation}
\begin{split}
\dpart{W^{\circ}}{t}(x,t)&=\dpart{\ln(z^{\circ})}{x}(x,t),\\
\dpart{W^{\bullet}}{t}(y,t)&=-\dpart{\ln(z^{\bullet})}{x}(y,t).
\end{split}
\end{equation}
\end{prop}

These expressions can be deduced from Proposition \ref{prop:dW/dt} by a complex analytic argument: we extract 
the coefficient of $z^0$ by a contour integral and apply the residue theorem. Interestingly, the combinatorial 
 approach taken in \cite{Albenque2026}, and based on \cite{BanFla2002}, is to interpret the generating function
$\dpart{W^{\circ}(x,t)}{t}$ (resp. $\dpart{W^{\bullet}(y,t)}{t}$) as the generating function 
counting walks starting and ending at zero, such that 
each step of height $i$ is associated with a weight $a_i=[z^{-i}]x(z^{-1})$ (resp. $b_i=[z^i]y(z)$).
 In this frame, the quantity $z^{\circ}(x)$ (resp. $z^{\bullet}(y)$) is the generating 
functions of excursions, or positive walks starting and ending at $0$.

\subsubsection{Poisson formula for the parametrization $(x(z),y(z))$}
Let us now examine an important formula, which first appeared in \cite{Eynard2003,Eynard2002} in the case of uncoloured planar maps, and which we shall here generalize.

\begin{lem}[Poisson formula for $x(z)$ and $y(z)$]\label{lem:poisson}
The following formula holds for the generating functions $x(z)$ and $y(z)$ of elementary slices:
\begin{equation}
\dpart{x}{z}(z,t)\,\dpart{y}{t}(z,t)-\dpart{y}{z}(z,t)\,\dpart{x}{t}(z,t)=\frac{1}{z}.
\end{equation}
\end{lem}

\begin{proof}

The first step of the proof comes from Propositions \ref{defth:z0} and \ref{prop:monochromatic}, which allows us to express $y(z,t)$ as a function of $x(z,t)$. Indeed, one can write
$$
y(z,t)=y(z^{\circ}(x(z,t),t),t)=Y(x(z,t),t).
$$  
By differentiating this relation using the chain rule, we obtain the two equations:
\begin{align*}
\dpart{y}{z}(z,t)&=\dpart{x}{z}(z,t)\cdot \dpart{Y}{x}(x(z,t),t),\\
\dpart{y}{t}(z,t)&=\dpart{x}{t}(z,t)\cdot \dpart{Y}{x}(x(z,t),t)+\dpart{Y}{t}(x(z,t),t).
\end{align*}

Substituting these two $y$-derivatives into the Poisson bracket $\dpart{x}{z}(z,t)\,\dpart{y}{t}(z,t)-\dpart{y}{z}(z,t)\,\dpart{x}{t}(z,t)$, which we wish to simplify, gives:
$$
\dpart{x}{z}(z,t)\,\dpart{y}{t}(z,t)-\dpart{y}{z}(z,t)\,\dpart{x}{t}(z,t)
=\dpart{Y}{t}(x(z,t),t)\cdot\dpart{x}{z}(z,t).
$$

Recalling that $\dpart{Y}{t}(x,t)=\dpart{W^{\circ}}{t}(x,t)$, we can simplify the right-hand side using Proposition \ref{prop:dmonochromatic}, which provides a formula for this partial derivative. We thus find:
$$
\dpart{Y}{t}(x(z,t),t)=\dpart{W^{\circ}}{t}(x(z,t),t)=\frac{\dpart{z^{\circ}}{x}(x(z,t),t)}{z}.
$$
Using the relation $z^{\circ}(x(z,t),t)=z$ from Proposition \ref{defth:z0} and differentiating it with respect to $z$ in order to obtain $\dpart{z^{\circ}}{x}(x(z,t),t)=\frac{1}{\dpart{x}{z}(z,t)}$,  
we finally get
$$
\dpart{x}{z}(z,t)\,\dpart{y}{t}(z,t)-\dpart{y}{z}(z,t)\,\dpart{x}{t}(z,t)
=\frac{\dpart{z^{\circ}}{x}(x(z,t),t)}{z^{\circ}(x(z,t),t)}\,\cdot\dpart{x}{z}(z,t)=\frac{1}{z}.
$$
This concludes the proof.
\end{proof}

It would be interesting to have a bijective proof of Lemma \ref{lem:poisson}.

\subsubsection{Application: an expression of $A_1$ as a $t$-derivative}

Now that we have all the necessary tools at our disposal, it remains only to combine them in such a way as to remove the integral with respect to $t$ in Equation \eqref{eq:dH1/dt}.  
This is because it is difficult to interpret the algebraic nature of primitives, with respect to $t$, of our generating functions.  
We shall therefore remedy this by showing that $A_1(x;y,t)$ can be expressed as a derivative with respect to $t$.  
\begin{prop}\label{prop:a1dt}
The generating function of accessibly 
pointed $1$-alternating hypermaps can be written as
\begin{equation}
A_1(x;y,t)=\frac{\partial}{\partial t}[z^0]\,z\ln\left(1-\frac{y(z,t)}{y}\right)\frac{\partial}{\partial z}\ln\left(1-\frac{x(z,t)}{x}\right).
\end{equation}
\end{prop}
This equation, completed with Equation \eqref{eq:dH1/dt}, yields the last equality of Equation \eqref{eq:a1h1b1rel} in Proposition \ref{prop:k1}. Indeed, 
the remaining step is to check that both sides evaluate to $1$ at $t=0$. In Appendix \ref{appendix}, we prove that $[t^0]x(z,t)=V_{\bullet}'(z^{-1})\in \R[z^{-1}]$ and $[t^0]y(z,t)=z^{-1}\in z^{-1}\R[z^{-1}]$, which implies that the right-hand side 
indeed evaluate to $1$ as $t=0$. Now let us prove Proposition \ref{prop:a1dt}.

\begin{proof}
We start from the left-hand side by writing $A_1(x;y,t)$ in the form 
$[z^0]\frac{1}{(x-x(z,t))(y-y(z,t))}$ as given by Theorem \ref{thm:ak} for $k=1$.  
We then modify the numerator $1$ using Lemma \ref{lem:poisson} through the relation:
$$
1=z\,\left(\dpart{x}{z}(z,t)\,\dpart{y}{t}(z,t)-\dpart{y}{z}(z,t)\,\dpart{x}{t}(z,t)\right).
$$
We can now proceed to the computation. Splitting the sum into two parts, we obtain:
$$
[z^0]\frac{1}{(x-x(z,t))(y-y(z,t))}=[z^0]\,z\left(\frac{\dpart{x}{z}(z,t)\,\dpart{y}{t}(z,t)}{(x-x(z,t))(y-y(z,t))}-\frac{\dpart{x}{t}(z,t)\,\dpart{y}{z}(z,t)}{(x-x(z,t))(y-y(z,t))}\right).
$$
Introducing logarithmic derivatives, we can rewrite the right-hand side as:
\begin{multline*}
[z^0]\,z\left(\frac{\partial}{\partial z}\ln(x-x(z,t))\cdot\frac{\partial}{\partial t}\ln(y-y(z,t))
-\frac{\partial}{\partial t}\ln(x-x(z,t))\cdot\frac{\partial}{\partial z}\ln(y-y(z,t))\right)=\\
\sum_{h\in\Z}h[z^h]\ln(x-x(z,t))\,[z^{-h}]\frac{\partial}{\partial t}\ln(y-y(z,t))
-\sum_{h\in\Z}h[z^{-h}]\frac{\partial}{\partial t}\ln(x-x(z,t))\,[z^{h}]\ln(y-y(z,t)).
\end{multline*}
We then change $h$ to $-h$ in the second sum, which reveals the derivative of a product:
$$
[z^0]\frac{1}{(x-x(z,t))(y-y(z,t))}=\frac{\partial}{\partial t}\sum_{h\in\Z}h[z^h]\ln(x-x(z,t))\,[z^{-h}]\ln(y-y(z,t)).
$$
This relation can be rewritten as:
$$
A_1(x;y)=\frac{\partial}{\partial t}[z^0]\,z\ln\left(1-\frac{y(z,t)}{y}\right)\frac{\partial}{\partial z}\ln\left(1-\frac{x(z,t)}{x}\right),
$$
which is exactly what we wanted to show.
\end{proof}

\subsubsection{A relation between $A_1$ and the spectral curve}\label{sec:a1E}

As introduced in Section \ref{sec:monochromatic}, hypermaps with a monochromatic boundary play a crucial role in the combinatorial interpretation of various algebraic quantities such as $X(y)$ and $Y(x)$, hence in the understanding of the spectral curve $E(x,y)=0$.  
The latter originally appeared in the context of solving Tutte's equations, that is, by removing the marked edge of the outer face in order to derive a closed form for $H_1$. This is how the following relation appeared in \cite{Eynard2002}:
\begin{equation}\label{eq:h1E}
H_1(x;y)=\frac{E(x;y)}{(x-X(y))(y-Y(x))}.
\end{equation}
In fact, using Equation \eqref{eq:a1h1b1rel}, one can rewrite this equation as an equation on 
the generating function of accessibly pointed $1$-alternating hypermaps:
$$
A_1(x;y)=\frac{\partial}{\partial t}\ln\left(\frac{E(x,y)}{(x-X(y))(y-Y(x))}\right).
$$
The goal of this section is to prove this equation, which is the main result of Proposition \ref{prop:a1E},
thereby giving another proof of Equation \eqref{eq:h1E} without using Tutte's method.
\\

First, let us denote by $d^{\circ}$ (resp. $d^{\bullet}$) the maximal degree of white (resp. black) faces.  
In this case, the difference between the left and right sides of an elementary slice is bounded. We have already seen that it is bounded by $-1$ from below but, going around 
the face incident to the base edge of such a slice, we find an upper bound of $d^{\bullet}-1$ for slices of type $\mathcal{A}$ and of $d^{\circ}-1$ for slices of type $\mathcal{B}$, since the left and right sides of the slice are geodesics.  
We therefore conclude that $z^{d^{\bullet}-1}(x(z)-x)$ is a polynomial in $z$ of degree at most $d^{\bullet}$, while $z(y(z)-y)$ is a polynomial of degree at most $d^{\circ}$.  
We may thus consider their resultant. This resultant is the minimal polynomial whose zeros are parametrized by $(x(z),y(z))$, hence it is proportional to $E(x,y)$. Note that the definition of $E$ assumes that the proportionality constant does not depend on $t$, so it 
will not play any role, since its logarithmic derivative with respect to $t$ vanishes. From now on, we will say that $E$ is exactly the resultant of $(x(z)-x,y(z)-y)$, for simplicity.
\\
Since $E(x,y)$ is a resultant, we shall express it using all the roots of $x(z)-x$ and $y(z')-y$.  
We have already introduced one root $z^{\circ}$ and $z^{\bullet}$ for each of the two Laurent polynomials 
(which are in fact compositional inverses of each other); we now introduce the remaining ones.
To this end, we apply the Newton--Puiseux theorem, which ensures that there exist $d^{\bullet}-1$ other roots of the equation $x(z)=x$, denoted $z_j^{\circ}(x)$ for $1\leqslant j\leqslant d^{\bullet}-1$, such that 
$$
z_j^{\circ}(x)\in e^{\frac{2ij\pi}{d^{\bullet}-1}}\left(\frac{x^{-1}}{a_{d^{\bullet}-1}}\right)^{\frac{1}{d^{\bullet}-1}}\,\left(1+x^{-\frac{1}{d^{\bullet}-1}}\K[\![x^{-\frac{1}{d^{\bullet}-1}}]\!]\right),
$$
where $\K$ denotes the field of fractions of $\R$. In particular, an interesting fact on the construction of these roots is that they are all distinct.  
Similarly, we introduce the other roots $z_j^{\bullet}(y)$ for $1\leqslant j\leqslant d^{\circ}-1$ of $y(z)=y$, satisfying 
$$
z_j^{\bullet}(y)\in e^{\frac{2ij\pi}{d^{\circ}-1}}\left(\frac{y^{-1}}{b_{-1}}\right)^{-\frac{1}{d^{\circ}-1}}\,\left(1+y^{-\frac{1}{d^{\circ}-1}}\K[\![y^{-\frac{1}{d^{\circ}-1}}]\!]\right).
$$
In what follows, we will call $z^{\bullet}(y)$ and all the $z_j^{\circ}(x)$ the \emph{small roots} of the equations $y(z')=y$ and $x(z)=x$, as they tend to $0$ as $x,y$ goes to $\infty$. 
Similarly, the other roots will be called the \emph{large roots} since they tend to $\infty$.
\\

We now have all the ingredients needed to prove Equation \eqref{eq:a1E}.  
Let us begin by expanding the expression of the resultant $E(x,y)$.  
Since the resultant of two Laurent polynomials $P$ and $Q$ is (proportional to) $\prod_{i}Q(\alpha_i)$, where the $\alpha_i$ are the roots of $P$, we obtain the following explicit formula, with $Q(z):=x-x(z)$ and $P(z):=y-y(z)$:
$$
E(x,y)=(x-X(y))\prod_{i=1}^{d^{\circ}-1}(x-x(z_i^{\bullet}(y))),
$$
where $X(y)=x(z^{\bullet}(y))$. As explained earlier, the proportionality constant plays no role here, so we set it to $1$.
The latter relation can also be written as
\begin{equation}\label{eq:resultant}
    \frac{E(x,y)}{(x-X(y))(y-Y(x))}=\frac{\prod_{j=1}^{d^{\circ}-1}(x-x(z_j^{\bullet}(y)))}{y-Y(x)}.
\end{equation}
Taking the logarithm and differentiating then gives:
\begin{equation}\label{eq:dlnE/dt}
\frac{\partial}{\partial t}\ln\left(\frac{E(x,y)}{(x-X(y))(y-Y(x))}\right)=\frac{\frac{\mathrm{d}}{\mathrm{d}t}y(z^{\circ}(x,t),t)}{y-y(z^{\circ}(x))} - \sum_{j=1}^{d^{\bullet}-1}\frac{\frac{\mathrm{d}}{\mathrm{d}t}x(z_j^{\bullet}(y,t),t)}{x-x(z_j^{\bullet}(y))}.
\end{equation}

Now that we have sufficiently developed the right-hand side of Equation \eqref{eq:a1E}, let us turn back to the left-hand side.  
Recall from Theorem \ref{thm:ak} that we have an explicit form for the generating function of accessibly pointed $1$-alternating maps.  
The next step consists in performing a partial fraction decomposition with respect to $z$ of this expression, before extracting the $z^0$ coefficient.  
We thus write:
\begin{align*}
A_1(x;y)&=[z^0]\frac{1}{(x-x(z))(y-y(z))}\\&=[z^0]\frac{z^{d^{\bullet}}}{C(z-z^{\circ}(x))\prod_{j=1}^{d^{\circ}-1}(z-z_j^{\bullet}(y))\cdot(z-z^{\bullet}(y))\prod_{j=1}^{d^{\bullet}-1}(z-z_j^{\circ}(x))},
\end{align*}
where we place the large roots on the left and the small roots on the right of the product. 
The constant $C$, independent of $z$, will be irrelevant for our present discussion.  
\\
Performing the partial fraction decomposition leads to:
\begin{align*}
A_1(x;y)&=[z^0]\Bigg(\frac{1}{z-z^{\circ}(x)}\frac{(z^{\circ}(x))^{d^{\bullet}}}{\dpart{z^{d^{\bullet}}(x-x(z))(y-y(z))}{z}_{|z=z^{\circ}(x)}}+\sum_{j=1}^{d^{\circ}-1}\frac{1}{z-z_j^{\bullet}(y)}\frac{(z_j^{\bullet}(y))^{d^{\bullet}}}{\dpart{z^{d^{\bullet}}(x-x(z))(y-y(z))}{z}_{|z=z_j^{\bullet}(y)}}\\
&+\frac{1}{z-z^{\bullet}(y)}\frac{(z^{\bullet}(y))^{d^{\bullet}}}{\dpart{z^{d^{\bullet}}(x-x(z))(y-y(z))}{z}_{|z=z^{\bullet}(y)}}+\sum_{j=1}^{d^{\bullet}-1}\frac{1}{z-z_j^{\circ}(x)}\frac{(z_j^{\circ}(x))^{d^{\bullet}}}{\dpart{z^{d^{\bullet}}(x-x(z))(y-y(z))}{z}_{|z=z_j^{\circ}(x)}}\Bigg).
\end{align*}

This expression simplifies slightly when differentiating the denominator. Indeed, let us simplify the first denominator:
\begin{align*}
\dpart{z^{d^{\bullet}}(x-x(z))(y-y(z))}{z}_{|z=z^{\circ}(x)}=&\underbrace{(x-x(z^{\circ}(x)))}_{=0}\dpart{z^{d^{\bullet}}(y-y(z))}{z}_{|z=z^{\circ}(x)}\\
&+(z^{\circ}(x))^{d^{\bullet}}(y-y(z^{\circ}(x)))\dpart{(x-x(z))}{z}_{|z=z^{\circ}(x)}\\
=&-(z^{\circ}(x))^{d^{\bullet}}(y-y(z^{\circ}(x)))\dpart{x}{z}_{|z=z^{\circ}(x)},
\end{align*} 
and the three other $z$-derivatives simplify the same way. This yields
\begin{multline*}
A_1(x;y)=[z^0]\Bigg(\frac{1}{z-z^{\circ}(x)}\frac{-1}{(y-y(z^{\circ}(x)))\dpart{x}{z}_{|z=z_j^{\circ}(x)}}+\sum_{j=1}^{d^{\circ}-1}\frac{1}{z-z_j^{\bullet}(y)}\frac{-1}{(x-x(z_j^{\bullet}(y)))\dpart{y}{z}_{|z=z_j^{\bullet}(y)}}\\
+\frac{1}{z-z^{\bullet}(y)}\frac{-1}{(x-x(z^{\bullet}(y)))\dpart{y}{z}_{|z=z^{\bullet}(y)}}+\sum_{j=1}^{d^{\bullet}-1}\frac{1}{z-z_j^{\circ}(x)}\frac{-1}{(y-y(z_j^{\circ}(x)))\dpart{x}{z}_{|z=z_j^{\circ}(x)}}\Bigg).
\end{multline*}

We now extract the coefficient of $z^0$ from each term.  
Since the partial fraction decomposition was performed in a splitting field for $(x-x(z))(y-y(z))$, recall that $A_1(x;y)$ belongs to $x^{-1}y^{-1}\R[\![x^{-1},y^{-1}]\!]$.  
Thus, we are working in an algebraic extension of this ring, namely the field of Puiseux series in $x^{-1}$ and $y^{-1}$ with coefficients in $\K$.  
This means that $\frac{1}{z-z^{\circ}(x)}$ should be expanded as a power series in $z$, since its $z^n$ coefficient (for $n\geqslant0$) is then $z^{\circ}(x)^{-n-1}$, a power series in $x^{-1}$. The same reasoning applies to the other small roots, so the fraction $\frac{1}{z-z_j^{\bullet}(y)}$, $1\leqslant j\leqslant d^{\circ}-1$, is expanded in the same way.  
In contrast, in the case of the large roots, $\frac{1}{z-z^{\bullet}(y)}$ and all the fractions $\frac{1}{z-z_j^{\circ}(x)}$, $1\leqslant j\leqslant d^{\bullet}-1$, need to be expanded as Laurent power series in $z^{-1}$, so that their terms are series in $y^{-1}$.  
Hence we need to write these fractions as $\frac{z^{-1}}{1-z^{-1}z^{\bullet}(y)}$ and expand for large $z$, which finally have no $z^0$-coefficient. So we write
\begin{align*}
A_1(x;y)&=[z^0]\Bigg(\frac{1}{z-z^{\circ}(x)}\frac{-1}{(y-y(z^{\circ}(x)))\dpart{x}{z}_{|z=z_j^{\circ}(x)}}+\sum_{j=1}^{d^{\circ}-1}\frac{1}{z-z_j^{\bullet}(y)}\frac{-1}{(x-x(z_j^{\bullet}(y)))\dpart{y}{z}_{|z=z_j^{\bullet}(y)}}\\
&+\frac{z^{-1}}{1-z^{-1}z^{\bullet}(y)}\frac{-1}{(x-x(z^{\bullet}(y)))\dpart{y}{z}_{|z=z^{\bullet}(y)}}+\sum_{j=1}^{d^{\bullet}-1}\frac{z^{-1}}{1-z^{-1}z_j^{\circ}(x)}\frac{-1}{(y-y(z_j^{\circ}(x)))\dpart{x}{z}_{|z=z_j^{\circ}(x)}}\Bigg)\\
&=  \frac{1}{z^{\circ}(x)}\frac{1}{(y-y(z^{\circ}(x)))\dpart{x}{z}_{|z=z_j^{\circ}(x)}}+\sum_{j=1}^{d^{\circ}-1}\frac{1}{z_j^{\bullet}(y)}\frac{1}{(x-x(z_j^{\bullet}(y)))\dpart{y}{z}_{|z=z_j^{\bullet}(y)}}\\
&\overset{Y(x)=y(z^{\circ}(x))}{=}  \frac{1}{z^{\circ}(x)\,\dpart{x}{z}_{|z=z_j^{\circ}(x,t)}}\frac{1}{(y-Y(x))}+\sum_{j=1}^{d^{\circ}-1}\frac{1}{z_j^{\bullet}(y)\,\dpart{y}{z}_{|z=z_j^{\bullet}(y)}}\frac{1}{(x-x(z_j^{\bullet}(y)))}.
\end{align*}

Finally, the following lemma allows us to relate this expression to Equation \eqref{eq:dlnE/dt}, which immediately yields the desired result.

\begin{lem}\label{lem:dx(z0)}
    We have the following formulas for derivatives of compositions:
    $$
    \frac{\mathrm{d}}{{\mathrm{d}t}}y(z^{\circ}(x,t),t)=\frac{1}{z^{\circ}(x)\dpart{x}{z}(z^{\circ}(x))}\,\,\,\text{and}\,\,\,\forall j\in[\![1,d^{\circ}-1]\!],  \frac{\mathrm{d}}{{\mathrm{d}t}}x(z_j^{\bullet}(x,t),t)=-\frac{1}{z_j^{\bullet}(y)\dpart{y}{z}(z_j^{\bullet}(y))}.
    $$
\end{lem}

\begin{proof}
    Since the proof is the same for all formulas, we will only discuss the first one.

    We start by applying the chain rule on the left-hand side: $$\frac{\mathrm{d}}{{\mathrm{d}t}}y(z^{\circ}(x,t),t)=\dpart{y}{t}(z^{\circ}(x,t),t) + \dpart{z^{\circ}}{t}(y,t)\cdot\, \dpart{y}{t}(z^{\circ}(x,t),t).$$ Hence we are led to make $\dpart{z^{\circ}}{t}(x,t)$ explicit.

    To this end, the idea is to differentiate with respect to $t$ the equation $x(z^{\circ}(x,t),t)=x$. This gives
    $$
    \dpart{z^{\circ}}{t}(x,t)=-\frac{\dpart{x}{t}(z^{\circ}(x,t),t)}{\dpart{x}{z}(z^{\circ}(x,t),t)}.
    $$
    We now only need to substitute this relation into the former one and conclude using Lemma \ref{lem:poisson}:
    \begin{align*}
    \frac{\mathrm{d}}{\mathrm{d}t}y(z^{\circ}(x,t),t)&=\dpart{y}{t}(z^{\circ}(x)) + \left(-\frac{\dpart{x}{t}(z^{\circ}(x))}{\dpart{x}{z}(z^{\circ}(x))}\right)\cdot\, \dpart{y}{t}(z^{\circ}(x))\\
    &=\left(\dpart{x}{z}\,\dpart{y}{t}-\dpart{y}{z}\,\dpart{x}{t}\right)_{|z=z^{\circ}(x)}\cdot\frac{1}{\dpart{x}{z}(z^{\circ}(x))}\\
&=\frac{1}{z^{\circ}(x)\dpart{x}{z}(z^{\circ}(x))}
    \end{align*}
    which concludes the proof of the lemma and hence the proof of Proposition \ref{prop:a1E}.

\end{proof}

\section{Enumeration of $2$-alternating hypermaps}
\label{sec:2alt}
In this section, we establish Theorem \ref{thm:h2b2}.  
Our strategy relies on most of the concepts introduced earlier: we decompose a pointed $2$-alternating hypermap according to the component containing the marked vertex and all the vertices that can reach it, called the \emph{accessible component}, and the remaining part.  
We shall see that this decomposition follows naturally from a planarity argument, which allows us to separate the hypermap in a particularly straightforward way.

We then apply this result—namely Equation \eqref{eq:h2}—to derive the generating function of strongly connected $2$-alternating hypermaps, which is Equation \eqref{eq:b2}.  
This derivation does not explicitly rely on the notion of accessibility, although a similar argument could easily be adapted to the case of $B_2$.

\subsection{Expressing $2$-alternating hypermaps with accessibly pointed hypermaps}

Our approach follows the same direction as in the previous section during the proof of Equation \eqref{eq:dH1/dt}, with one essential difference:  
we shall now investigate whether the outward corners are connected to the distinguished vertex of the hypermap or not.  
Note that, in the case of $1$-alternating hypermaps, the marked vertex was always accessible from the outward corner, but here it will be a central element of our discussion.  

The goal of this section is therefore to prove the following lemma:
\begin{lem}\label{lem:dH2/dt}
A pointed $2$-alternating hypermap can be decomposed as follows:
\begin{equation}\label{eq:dH2/dt}
\begin{split}
\dpart{H_2(x_1,x_2;y_1,y_2)}{t}
&= A_2(x_1,x_2;y_1,y_2)\,H_1(x_1;y_1)\,H_1(x_2;y_2)\\
&+ \left(A_1(x_1;y_2)+A_1(x_2;y_1)\right)\,H_2(x_1,x_2;y_1,y_2).
\end{split}
\end{equation}
\end{lem}

Let us first see how this formula allows us to obtain Equation \eqref{eq:h2} of Theorem \ref{thm:h2b2}.  
To do so, we write, using Equation \eqref{eq:dH1/dt}
\begin{align*}
\dpart{}{t}\left(\frac{H_2(x_1,x_2;y_1,y_2)}{H_1(x_1;y_2)\,H_1(x_2;y_1)}\right)=&\frac{1}{H_1(x_1;y_2)\,H_1(x_2;y_1)}\Bigg(
\dpart{H_2(x_1,x_2;y_1,y_2)}{t}\\
&\hspace{2cm}-(A_1(x_1;y_2)+A_1(x_2;y_1))\,H_2(x_1,x_2;y_1,y_2)\Bigg).
\end{align*}
We then deduce, by substituting Equation \eqref{eq:dH2/dt}:
\begin{equation}\label{eq:dH2/H1H1/dt}
\dpart{}{t}\left(\frac{H_2(x_1,x_2;y_1,y_2)}{H_1(x_1;y_2)\,H_1(x_2;y_1)}\right)
= A_2(x_1,x_2;y_1,y_2)\,
\frac{H_1(x_1;y_1)\,H_1(x_2;y_2)}{H_1(x_1;y_2)\,H_1(x_2;y_1)}.
\end{equation}
We can now use Equation \eqref{eq:A=decompo2} from Corollary \ref{cor:ak}.
This allows us to recognize a total derivative on the right-hand side of Equation \eqref{eq:dH2/H1H1/dt}:
$$
A_2(x_1,x_2;y_1,y_2)\,
\frac{H_1(x_1;y_1)\,H_1(x_2;y_2)}{H_1(x_1;y_2)\,H_1(x_2;y_1)}
=\frac{1}{(x_1-x_2)(y_1-y_2)}\,
\dpart{}{t}
\left(\frac{H_1(x_1;y_1)\,H_1(x_2;y_2)}{H_1(x_1;y_2)\,H_1(x_2;y_1)}\right),
$$
by using the logarithmic derivative expression of $A_1$ in Equation \eqref{eq:dH1/dt}.  
We then conclude by integrating and using the facts that $[t^0]\,H_1(t)=1$ and $[t^0]\,H_2(t)=0$ by definition of $H_k$ in Equation \eqref{eq:Hkdef}, which gives Equation \eqref{eq:h2}.
\\

We can now get back to the combinatorial proof of Equation \eqref{eq:dH2/dt}.

\begin{proof}[Proof of Lemma \ref{lem:dH2/dt}]
\begin{figure}[t]
    \centering
    \includegraphics[width=1\textwidth]{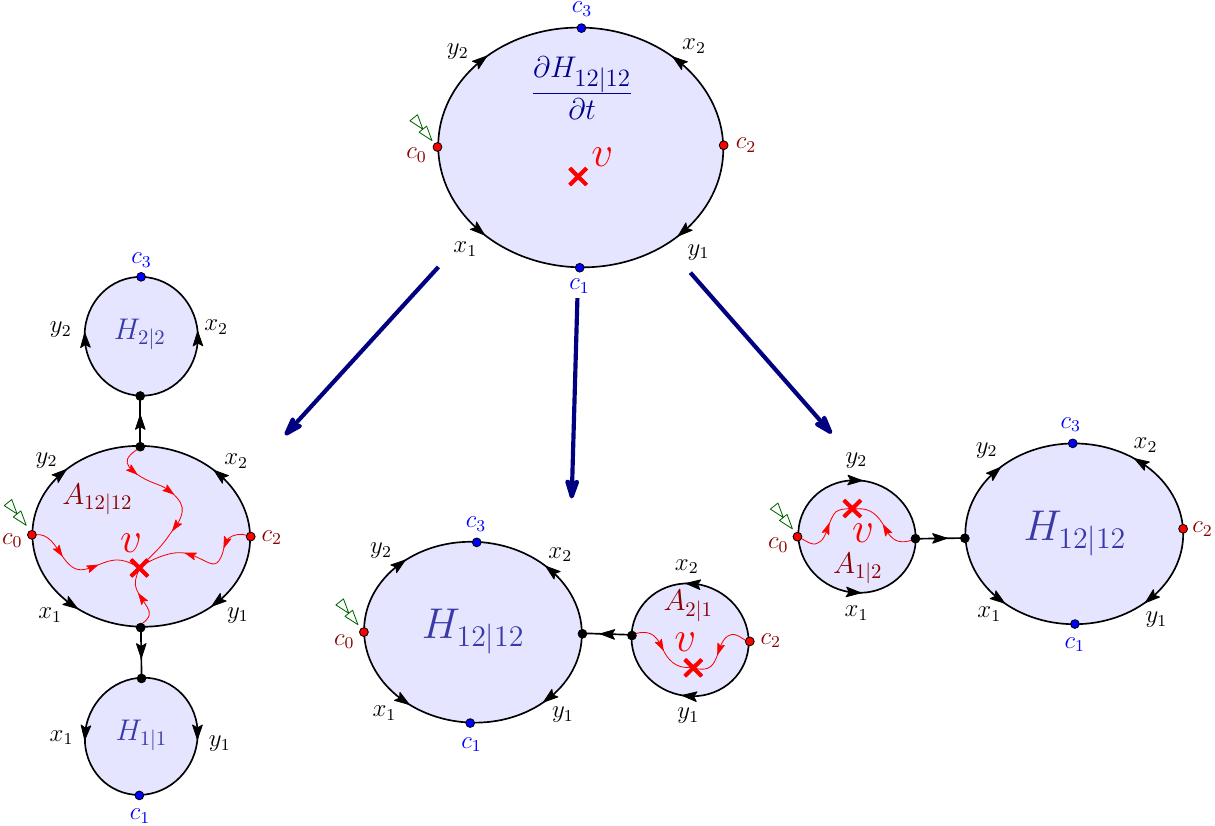} 
    \caption{Sketch of the proof of Lemma \ref{lem:dH2/dt}, from left to right the cases $1,2$ and $3$ according to whether $c_0$ or $c_2$ is related to $v$. We use here the shorthand notation $H_{i\mid j}:=H_1(x_i;y_j)$ and $H_{12\mid12}:=H_2(x_1,x_2;y_1,y_2)$ for simplicity.}
    \label{fig:k2_preuve}
\end{figure}
  
We begin the proof by considering a pointed $2$-alternating hypermap, denoting by $v$ its marked vertex.  
We shall also use the corner notation $c_0,c_1,c_2,c_3$ introduced in Section \ref{sec:hyperdef}.  
We focus on the even-indexed corners $c_0$ and $c_2$, since they are the ones most likely to access to $v$.
\\
Let us distinguish cases according to which of the corners $c_0$ and $c_2$ is connected to $v$. Figure \ref{fig:k2_preuve} gives a sketch of the proof.
    \\

    $\hookrightarrow$ \textit{Case 1: $v$ is accessible from both $c_0$ and $c_2$.} In this case only $c_1$ and $c_3$ may fail to access to $v$: there may exist 
    a bridge on their way to $v$. Taking the last such bridge for each of them isolates $c_1$ and $c_3$ 
    into one $1$-alternating hypermap each, of respective weight $H_1(x_1;y_1)$ and $H_1(x_2;y_2)$. Note the presence of the additional $1$ in $H_1$, which stands for 
    the case where there is no bridge preventing $c_1$ or $c_3$ from being related to $v$. What remains is an accessibly pointed $2$-alternating hypermap component containing $v,c_0$ and $c_2$, counted by $A_2(x_1,x_2;y_1,y_2)$.
    \\
    $\hookrightarrow$ \textit{Case 2: $v$ is accessible from $c_2$ but not from $c_0$.} Since $c_0$ is not related to $v$, there must be a 
    bridge oriented from $v$ to $c_0$. Looking at the orientations of the boundary edges of the 
    hypermap, this bridge cannot lie on $[c_3,c_1]$: otherwise, because of the orientation of the bridge, this boundary interval would have to 
    contain an inward corner, hence a marked corner with odd index by Remark \ref{rem:inwardcorner}, which is not possible. So the bridge must belong to the boundary interval $[c_1,c_3]$. 
    Taking the last of these bridges, we can split the pointed hypermap into two components: the first is not pointed and contains $c_0,c_1,c_3$, while the second, containing $c_2$ and $v$, is a pointed $1$-alternating hypermap. The latter is also accessible since we took the last bridge, and $c_2$ must be an outward corner because the orientation of the bridge forces the existence of an outward corner belonging to $[c_1,c_3]$, which is then $c_2$.
    \\
    $\hookrightarrow$ \textit{Case 3: $v$ is accessible from $c_0$ but not from $c_2$.} By symmetry with Case 2, the corresponding term is $H_2(x_1,x_2;y_1,y_2)\,A_1(x_1;y_2)$.
    \\
    Note that the obstruction of $c_0$ from $v$ implies that $c_2$ can reach $v$, and reciprocally. Hence the case where neither $c_0$ nor $c_2$ can reach $v$ does not occur.
    Finally, summing these three contributions gives exactly Equation \eqref{eq:dH2/dt}. 
\end{proof}

The idea of decomposing our proof according to which outward corners are related to the marked vertex stems from the fact that these are the boundary vertices most likely to be connected to any other vertex.  
Indeed, intuitively, inward corners have little chance of being connected to the marked vertex, since it suffices for a bridge formed by its two incident edges to obstruct such a connection—an event that is structurally inexpensive for the hypermap. It is also 
important to keep in mind that planarity and the uniqueness of the marked outer face are essential arguments here, since this enables us to decompose our $2$-alternating hypermaps into independent components.
These ideas form the foundation of the forthcoming generalization of Lemma \ref{lem:dH2/dt} in the upcoming work \cite{Lejeune2026}.

\subsection{On strongly connected $2$-alternating hypermaps}\label{sec:b2}

In order to study the generating series $B_2$ of strongly connected $2$-alternating hypermaps, we start from a generic $2$-alternating hypermap and ask what might prevent it from being strongly connected.  
As stated in Remark \ref{rem:inwardcorner}, it suffices to determine the conditions under which every vertex of the hypermap is accessible from the vertices incident to the corners $c_1$ and $c_3$. We illustrate the following reasonning in Figure \ref{fig:h2b2}.
\\
To that end, one can already see what could prevent corners $c_1$ and $c_3$ from being connected: the existence of a bridge belonging to $[c_0,c_2]$ or a bridge belonging to $[c_2,c_0]$.  
By considering the last such bridge on the path connecting $c_1$ to $c_3$, and then symmetrically on the path connecting $c_3$ to $c_1$ (which is possible since each bridge disconnects the hypermap), we observe the appearance of two components of weights $H_1(x_1;y_1)$ and $H_1(x_2;y_2)$, to which our corners respectively belong.  
We may therefore remove the two bridges and examine what remains, renaming the corners incident to the endpoints of each bridge $c_1$ and $c_3$ as in the original definition, but now assuming the absence of bridges on the paths from $c_1$ to $c_3$.  
\\
Let us now see how to connect $c_1$ and $c_3$ to $c_0$. This is possible only if there is no bridge belonging to $[c_3,c_1]$. Taking the first such bridge produces a component containing $c_0$ of weight $H_1(x_1;y_2)$.  
Symmetrically, for $c_2$, we obtain a component of weight $H_1(x_2;y_1)$.  
\\
By removing all these bridges, we can thus connect $c_1$ and $c_3$ to all boundary corners, and therefore every boundary vertex can be connected to any other boundary vertex.  
As a result, there cannot exist any remaining bridge in this component, by Lemma \ref{lem:accesscrit}.  
The remaining component is therefore strongly connected, as stated in Lemma \ref{lem:stronglyconnected}, and is thus counted by $B_2(x_1,x_2;y_1,y_2)$ once the bridge weights are added. Note that this decomposition is bijective, since one only needs to connect each $1$-alternating hypermap 
to the corresponding marked corner of the strongly connected $2$-alternating hypermap by a bridge.

 \begin{figure}[t]
    \centering
    \includegraphics[width=1\textwidth]{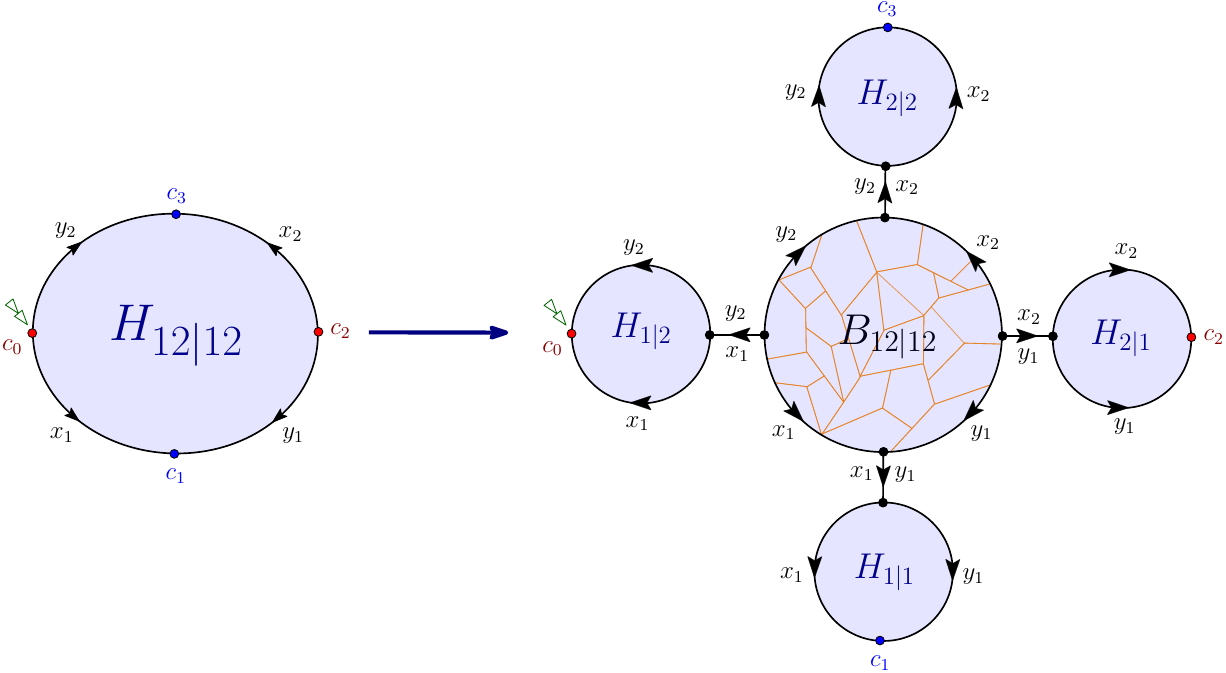} 
    \caption{Sketch of the decomposition of $H_2$ with respect to $B_2$. The idea is simply to remove all bridges that prevent each inner vertex from being connected to the marked corners, since the accessibility criterion of Lemma \ref{lem:accesscrit} relies on accessibility from the boundary. The shorthand notations in the figure are the same as in Figure \ref{fig:k2_preuve}.}
    \label{fig:h2b2}
\end{figure} 

This leads to the formula
\begin{equation}
H_2(x_1,x_2;y_1,y_2)
=H_1(x_1;y_1)\,H_1(x_1;y_2)\,H_1(x_2;y_1)\,H_1(x_2;y_2)\,B_2(x_1,x_2;y_1,y_2).
\end{equation}

This allows us to recover Equation \eqref{eq:b2} from Theorem \ref{thm:h2b2}, using the expression of $H_2(x_1,x_2;y_1,y_2)$ given in Equation \eqref{eq:h2}.  
This finally completes the proof of Theorem \ref{thm:h2b2}.

\section{Conclusion}

The main objective of this work was to highlight a new way of enumerating plane hypermaps with a non-monochromatic boundary.  
To achieve this, we have shown that the variable $t$, marking the vertices, plays a meaningful role.  
Indeed, marking a vertex gives rise to equations involving derivatives with respect to this variable.  
Moreover, studying the accessible component of the marked vertex—that is, the smallest hypermap in which all vertices can access the marked one—allows us to decompose pointed hypermaps.  
The rest of the work then consisted in understanding accessibly pointed hypermaps—made possible thanks to the notion of slices introduced in \cite{Albenque2026}—and in solving the resulting equations, either directly or by identifying a total derivative, in order to find algebraic expressions for alternating hypermaps.  
This approach enabled us to rederive several formulas appearing in \cite{Albenque2026} and \cite{Eynard2016}.
\\
The remaining open questions concern $k$-alternating hypermaps for $k \geqslant 3$—--a problem addressed in the forthcoming article \cite{Lejeune2026}—--as well as the extension to higher genus, in order to find new combinatorial interpretations of the algebraic results provided by topological recursion, as stated in \cite{EO2008}.  
Another direction for future work involves asymptotic considerations, which now appear more approachable thanks to a deeper understanding of the elementary building blocks—namely the generating functions of the elementary slices—bringing us closer to an understanding of distance distributions in hypermaps.

\appendix
\section{Counting plane hypermaps with a monochromatic boundary via integration}\label{appendix}

In this appendix, we give a new proof of Proposition~\ref{prop:monochromatic} using Proposition~\ref{prop:dW/dt} which itself follows from slice decomposition.
We shall first show that the derivatives with respect to $t$ of both sides of each equality in~\eqref{eq:monochromatic} are equal, before matching the integration constants.  
\\
We start from the case of pointed hypermaps with a white boundary, whose generating function is $\dpart{}{t}W^{\circ}(x,t)$.  
The first part of the following proof will be similar 
for pointed hypermaps with a black boundary.
\\

Our idea is to use the expression $\dpart{}{t}W^{\circ}(x,t)$ as $[z^0]\frac{1}{x-x(z,t)}$ and to perform a partial fraction decomposition.  
The issue is that this is only possible if we have finitely many roots for the equation $x=x(z,t)$, which occurs only when the face degrees are bounded. Let 
us show that doing it for any bound for the face degrees is enough to conclude in the 
general case, thanks to the following lemma.  

\begin{lem}\label{lem:boundeddegree}
Let $f$ be an element of $\R$, and let 
$$
f_d:=\Big[\prod_{i\geqslant d+1}(t_i^{\circ})^0\prod_{j\geqslant d+1}(t_j^{\bullet})^0\Big]f
$$
be the restriction of $f$ to faces with degree bounded by $d\geq1$.  
Then the identity $f=0$ holds if, and only if, $f_d=0$ holds for every $d\geq1$.  
\end{lem}
\begin{proof}
The direct implication is immediate by definition of $f_d$.

For the converse, recall that a formal series in infinitely many variables (as $f$ which belongs to $\R$) is defined as a sum of monomials of finite degree-that is, involving finitely many variables. Hence one only need 
to check that each coefficient from $f$ in front of any 
of these finite-degree monomials vanish, which is easy by considering the identities $f_d=0$ because they are exactly restricted to the case of finite (face) degree.
\end{proof}

The usefulness of this formal lemma lies in the fact that Proposition \ref{prop:monochromatic} makes sense in $\R$, while we wish to bound the degree of our faces (say, by $d^{\circ}$ for white faces and $d^{\bullet}$ for black faces) in order to speak of the roots of polynomials.  
However, this is merely a computational artifice, since the roots—besides the compositional inverses $z^{\circ}$ and $z^{\bullet}$—no longer appear in the final result.  
\\
We then denote all our roots as in Section \ref{sec:a1E}, and we can therefore write our partial fraction decomposition.  
We refer to the computation from Section \ref{sec:a1E}, since the same ideas apply: performing the decomposition over the roots by turning the denominator into a polynomial, which leads to
$$
\frac{1}{x-x(z)}=\frac{z^{d^{\bullet}}}{z^{d^{\bullet}}(x-x(z))}=\frac{1}{z-z^{\circ}(x)}\frac{(z^{\circ}(x))^{d^{\bullet}}}{\dpart{z^{d^{\bullet}}(x-x(z))}{z}_{|z=z^{\circ}(x)}}+\sum_{j=1}^{d^{\bullet}-1}\frac{1}{z-z_j^{\circ}(x)}\frac{(z_j^{\circ}(x))^{d^{\bullet}}}{\dpart{z^{d^{\bullet}}(x-x(z))}{z}_{|z=z_j^{\circ}(x)}}.
$$
This can be simplified by expanding the derivatives in the denominators as in Section~\ref{sec:a1E}:
$$
\frac{1}{x-x(z)}=-\frac{1}{z-z^{\circ}(x)}\frac{1}{\dpart{x}{z}_{|z=z^{\circ}(x)}}-\sum_{j=1}^{d^{\bullet}-1}\frac{1}{z-z_j^{\circ}(x)}\frac{1}{\dpart{x}{z}_{|z=z_j^{\circ}(x)}}.
$$
To extract the coefficient of $z^0$ as needed to compute $\dpart{}{t}W^{\circ}(x,t)$, only the first term contributes, since other terms are elements of $z^{-1}\R[\![z^{-1}]\!]$ because the partial fraction decomposition holds in $x^{-1}\R[\![x^{-1}]\!]$, and thus vanish.  
We finally obtain:
$$
[z^0]\frac{1}{x-x(z)}=\frac{1}{z^{\circ}(x)\,\dpart{x}{z}_{|z=z^{\circ}(x)}}.
$$
By Proposition~\ref{prop:dW/dt} and by Lemma \ref{lem:dx(z0)}, we then have:
$$
\dpart{}{t}W^{\circ}(x,t)=\frac{\mathrm{d}}{{\mathrm{d}t}}y(z^{\circ}(x,t),t)=\dpart{}{t}Y(x,t),
$$
since $Y(x,t)=y(z^{\circ}(x,t),t)$ by Equation \eqref{eq:XyYx}.  
By a dual argument for pointed hypermaps with a black boundary, we can show that
$$
\dpart{}{t}W^{\bullet}(y,t)=\dpart{}{t}X(y,t).
$$
Thanks to Lemma \ref{lem:boundeddegree}, the two last identities hold coefficientwise in $\R$, and not just in the case of bounded face degrees.
\\
It remains to perform the integration and so to determine the integration constants. We will then prove that 
$$
[t^0]W^{\circ}(x,t)-[t^0]Y(x,t)=-V_{\circ}'(x)
\quad\text{and}\quad
[t^0]W^{\bullet}(y,t)-[t^0]X(y,t)=-V_{\bullet}'(y).
$$

First, note that there exists no hypermap with a monochromatic boundary and zero vertex.  
Hence, we directly have $[t^0]W^{\circ}(x,t)=0$ and $[t^0]W^{\bullet}(y,t)=0$. So the problem reduces to showing that:
$$
[t^0]Y(x,t)=V_{\circ}'(x)
\quad\text{and}\quad
[t^0]X(y,t)=V_{\bullet}'(y).
$$

We will then use the expressions $X(y,t)=x(z^{\bullet}(y,t),t)$ and $Y(x,t)=y(z^{\circ}(x,t),t)$ from Proposition \ref{defth:z0}, from which we want to extract the $t^0$ coefficient.
To do so, we must study elementary slices having no vertices, except those on the right boundary (since they carry no weight as defined in Definition \ref{Def:sliceAB}).  
\\
$\hookrightarrow$ For slices of type $\mathcal{A}$, the vertex incident to the corner $l$ is weightless only if it is the apex.  
Since there are no internal vertices, we are restricted to the case where the base is incident to a single black face, whose remaining edges form the right boundary of the slice.  
Computing the increment in each case gives $[t^0]x(z,t)=\sum_{k\geqslant 1}t_k^{\bullet}z^{1-k}=V_{\bullet}'(z^{-1})$.  
This is a formal series in $z^{-1}$, which thus has no singular term in $z$.  
Hence, $[t^0]x(z^{-1},t)$ is not invertible, implying $[t^0]z^{\circ}(x,t)^{-1}=0$.  
See the proof of Proposition \ref{defth:z0} for a clearer understanding of how $z^{\circ}(x,t)$ is defined.  
\\
$\hookrightarrow$ For slices of type $\mathcal{B}$, identifying corners $l$ and $o$ restricts us to the trivial slice, consisting of one edge and two weightless vertices, with increment $-1$.  
Thus, $[t^0]y(z,t)=z^{-1}$.  
We then observe that $[t^0]z^{\bullet}(y,t)=y^{-1}$, since it is the unique solution of $[t^0]y(z,t)=y$.  
\\
These two observations allow us to conclude that $[t^0]z^{\bullet}(y,t)\neq0$ and:
$$
[t^0]X(y,t)=[t^0]x(z^{\bullet}(y,t),t)=\big([t^0]x(z,t)\big)_{|z=[t^0]z^{\bullet}(y,t)}=V_{\bullet}'(z^{-1})_{|z=y^{-1}}=V_{\bullet}'(y),
$$
which is what we wanted.  
\\

On the other hand, things are more subtle in order to compute $[t^0]Y(x,t)=[t^0]y(z^{\circ}(x,t),t)$ due to the definition of $z^{\circ}(x)$ of which we only have information about its multiplicative inverse, which is ill-defined near $t=0$ as explained earlier. The idea is to make $x(z)$ appear and compose it with $z^{\circ}(x)$ on its 
right to make simplifications.  
We use then Equations \eqref{eq:defxynew} and \eqref{eq:aibieq}:
$$
y(z^{\circ}(x,t),t)=(z^{\circ}(x,t))^{-1}+\big([z^{\geqslant0}]y(z,t)\big)_{|z=z^{\circ}(x,t)},
$$
where $[z^{\geqslant}0] f(z)$ stands for considering only the nonnegative powers of $z$ in the Laurent polynomial $f(z)$. The last term of the previous equation expands as
$$
\big([z^{\geqslant0}]y(z,t)\big)_{|z=z^{\circ}(x,t)}=\sum_{d\geqslant 1}t_d^{\circ}([z^{\geqslant0}]x(z,t)^{d-1})_{|z=z^{\circ}(x,t)}.
$$
Taking the $t^0$ coefficient and recalling that $[t^0]z^{\circ}(x,t)^{-1}=0$, we simplify:
$$
[t^0]y(z^{\circ}(x,t),t)=0+\sum_{d\geqslant 1}t_d^{\circ}[t^0]([z^{\geqslant0}]x(z,t)^{d-1})_{|z=z^{\circ}(x,t)}.
$$
Furthermore, since $([z^{<0}]x(z,t)^{d-1})_{|z=z^{\circ}(x,t)}$ is a polynomial in $z^{\circ}(x)$, then one gets the equality $[t^0]([z^{<0}]x(z,t)^{d-1})_{|z=z^{\circ}(x,t)}=0$ since $[t^0]z^{\circ}(x,t)=0$. We can therefore simplify the expression:
$$
[t^0]([z^{\geqslant0}]x(z,t)^{d-1})_{|z=z^{\circ}(x,t)}=[t^0](x(z,t)^{d-1})_{|z=z^{\circ}(x,t)}=x^{d-1},
$$
since $z^{\circ}(x)$ is the compositional inverse of $x(z)$.  
\\
We finally conclude that:
$$
[t^0]Y(x,t)=[t^0]y(z,t)_{|z=z^{\circ}(x,t)}=\sum_{d\geqslant 1}t_d^{\circ}([z^{\geqslant0}]x(z,t)^{d-1})_{|z=z^{\circ}(x,t)}=\sum_{d\geqslant 1}t_d^{\circ}x^{d-1}=V_{\circ}'(x),
$$
which is precisely what we wanted to prove.

\printbibliography

\end{document}